\documentclass{article}
\usepackage{arxiv}
\usepackage[T1]{fontenc}
\usepackage[english]{babel}
\usepackage{microtype}
\usepackage{amsmath,amssymb,amsthm,mathtools,aliascnt}
\usepackage{booktabs,tabularx,array}
\usepackage{graphicx,xcolor}
\usepackage{float,needspace}
\usepackage{enumitem}
\usepackage[colorlinks=true,allcolors=blue]{hyperref}
\usepackage[nameinlink,capitalise,noabbrev]{cleveref}
\crefname{appendix}{appendix}{appendices}
\Crefname{appendix}{Appendix}{Appendices}
\allowdisplaybreaks
\newtheorem{theorem}{Theorem}[section]
\newaliascnt{lemma}{theorem}
\newtheorem{lemma}[lemma]{Lemma}
\aliascntresetthe{lemma}
\Crefname{lemma}{Lemma}{Lemmas}
\newaliascnt{proposition}{theorem}
\newtheorem{proposition}[proposition]{Proposition}
\aliascntresetthe{proposition}
\Crefname{proposition}{Proposition}{Propositions}
\newaliascnt{corollary}{theorem}

\aliascntresetthe{corollary}
\Crefname{corollary}{Corollary}{Corollaries}
\theoremstyle{definition}
\newaliascnt{assumption}{theorem}
\newtheorem{assumption}[assumption]{Assumption}
\aliascntresetthe{assumption}
\Crefname{assumption}{Assumption}{Assumptions}
\newaliascnt{example}{theorem}
\newtheorem{example}[example]{Example}
\aliascntresetthe{example}
\Crefname{example}{Example}{Examples}
\theoremstyle{remark}
\newaliascnt{remark}{theorem}
\newtheorem{remark}[remark]{Remark}
\aliascntresetthe{remark}
\Crefname{remark}{Remark}{Remarks} \Crefname{assumption}{Assumption}{Assumptions}
\newcommand{\R}{\mathbb{R}}
\newcommand{\EE}{\mathbb{E}}
\newcommand{\EB}{\mathbb{E}_B}
\newcommand{\ED}{\mathbb{E}_D}
\newcommand{\PP}{\mathbb{P}}
\newcommand{\dd}{\,\mathrm{d}}
\newcommand{\ip}[2]{\langle #1,#2\rangle}
\newcommand{\fl}[1]{\underline{#1}_\Delta}

\title{Strong Convergence of Fully Discrete Tamed Euler Approximations for Time-Changed Stochastic Delay Differential Equations}
\author{Ruchun Zuo \quad Qian Guo\thanks{Corresponding author. Email: \href{mailto:qguo@shnu.edu.cn}{qguo@shnu.edu.cn}.}\\[0.5em]
{\small Department of Mathematics, Shanghai Normal University}\\
{\small Shanghai 200234, China}}
\date{}
\begin{document}
\maketitle
\begin{abstract}
We establish the strong convergence of a fully discrete equidistant tamed Euler--Maruyama (ETEM) method for time-changed stochastic delay differential equations (TCSDDEs) with state-dependent $\mathrm{d}t$ drifts.
A coupled one-sided condition on the subordinator-driven drift and diffusion allows superlinear coefficients with polynomial growth in both the current and delayed states.
For subordinators whose Laplace exponent satisfies $\psi(\lambda)\geq c_1\lambda^\alpha$ for all sufficiently large $\lambda$, with $\alpha\in(0,1)$, we derive strong error bounds in terms of the time step $\Delta$ and the subordinator sampling step $\delta$, without imposing a relationship between them.
Comparing the partitions generated by the exact and approximate inverse subordinators controls the sampling error without introducing a factor proportional to the number of time steps.
For coefficients Lipschitz in both time arguments and Lipschitz history, the scheme achieves a strong convergence rate of order $\Delta^{\alpha/2}(1+|\log\Delta|)^{(1-\alpha)/2}+\sqrt{\delta(1+|\log\delta|)}$.
A terminal lower bound $c(\Delta^{\alpha/2}+\delta^{1/2})$ for a multiplicative-noise equation with a stable subordinator establishes the sharpness of both polynomial exponents for this Euler scheme.
Numerical experiments illustrate convergence for nonautonomous and multidimensional superlinear delay equations, as well as delay equations with low-regularity time dependence.

\end{abstract}
\keywords{time-changed stochastic delay differential equation, inverse subordinator, fully discrete approximation, tamed Euler--Maruyama method, strong convergence, lower bound.}

\noindent\textbf{Mathematics Subject Classification (2020):} 65C30, 60H35, 34K50.

\section{Introduction}\label{sec:introduction}
Subdiffusion, characterized by a sublinear power-law scaling of the mean squared displacement over physical time, has emerged as a fundamental paradigm for describing non-Markovian transport across complex and heterogeneous media.
Ubiquitous manifestations appear across diverse scientific disciplines, including macromolecular crowding in cellular biology \cite{SaxtonJacobson1997}, tracer dispersion in disordered porous materials \cite{MetzlerKlafter2000}, and asset price movements with heavy-tailed resting periods in mathematical finance \cite{MagdziarzOrzelWeron2011}.
From the microscopic perspective of continuous-time random walks, such subdiffusive dynamics typically arise when particles experience trapping events, localized adhesion, or viscoelastic obstacles that induce heavy-tailed waiting times between successive jumps.
In the continuum diffusion limit, these microscopic trapping phenomena are effectively captured by stochastic differential equations (SDEs) driven by time-changed Brownian motion \cite{Kobayashi2011, MeerschaertScheffler2004}.
In this framework, physical Brownian fluctuations are subordinated to a non-decreasing, non-Markovian inverse subordinator whose flat plateaus precisely mirror the trapping intervals of the underlying particles.
Such time-changed formulations provide a rigorous stochastic representation for fractional Fokker--Planck equations and related anomalous transport equations \cite{CarnaffanKawai2017}.

In practical modeling, anomalous subdiffusion may also interact with hereditary memory, feedback loops, or transmission delays along the physical timeline.
To capture the interplay between non-Markovian trapping and delayed responses, time-changed stochastic delay differential equations (TCSDDEs) have recently attracted growing attention in qualitative stochastic analysis \cite{ZhangYuan2019, ZhangZhuYuan2021}.
From a modeling perspective, such formulations naturally distinguish two distinct dynamical features:
on the one hand, intrinsic microscopic fluctuations and trapping events advance with the subdiffusive clock and are driven by $\dd E(t)$ and $\dd B(E(t))$;
on the other hand, continuous background mechanisms—such as external forcing, steady dissipation, or transport along physical time—evolve continuously under $\dd t$.
For related subdiffusive models with external forces, see \cite{ChenWangDeng2019}.
Combining these dual driving forces with a discrete transmission delay $\tau>0$ to account for delayed state feedback, we consider on $[0,T]$ the following system
\begin{align}
    \mathrm{d}X(t) &=h(t,E(t),X(t),X(t-\tau))\dd t +f(t,E(t),X(t),X(t-\tau))\dd E(t)\notag\\
    &\quad+g(t,E(t),X(t),X(t-\tau))\dd B(E(t)),\qquad t\in[0,T],\notag\\
    X(t)&=\xi(t),\qquad t\in[-\tau,0].\label{eq:model}
\end{align}
where $E(t)$ is the inverse subordinator capturing the anomalous trapping dynamics, $B$ is an independent Brownian motion, and $\xi$ is a prescribed initial history.
In \eqref{eq:model}, the drift $h\dd t$ captures the continuous deterministic evolution of the system, while the subordinator-driven terms $f\dd E(t)$ and $g\dd B(E(t))$ account for the anomalous transport and stochastic fluctuations induced by the trapping mechanism.
The delayed state $X(t-\tau)$ accounts for hereditary feedback with fixed time lag $\tau$.
Qualitative properties such as pathwise existence, uniqueness, and stability for equations with time changes and delays have been investigated in \cite{ZhangYuan2019, ZhangZhuYuan2021}.
Subdiffusion-driven systems also arise in stochastic optimal control \cite{ZhangChen2024,zhang2026mean}.

Even when qualitative properties are established, explicit pathwise solutions to TCSDDEs are virtually nonexistent due to the interplay between non-Markovian trapping and infinite-dimensional delayed feedback.
For parameter estimation, risk quantification, and statistical inference in complex media, high-precision numerical simulation is indispensable \cite{MagdziarzOrzelWeron2011, MetzlerKlafter2000}.
In practical scientific computing, simulating time-changed systems across long horizons or large Monte Carlo ensembles places stringent demands on the numerical integrator.
While implicit methods offer reliable numerical stability for nonlinear stochastic systems \cite{HighamMaoStuart2002, MaoSabanis2003}, explicit integrators are strongly favored in multi-path settings due to their low per-step computational overhead.
However, in the presence of superlinear drifts and delayed feedback, classical explicit Euler--Maruyama schemes suffer from severe numerical explosion and finite-time moment divergence \cite{HutzenthalerJentzenKloeden2011,PetrovicMilosevic2026}, creating a sharp trade-off between computational cost and numerical stability.

Early numerical investigations for time-changed SDEs (TCSDEs) relied on the duality principle \cite{JumKobayashi2016}, which connects time-changed equations to classical ones.
However, this duality does not directly cover the independent $t$-dependence and state-dependent $\mathrm{d}t$ drift considered here.
Direct approximations on the subordinator grid, commonly known as subordinator-grid or $D$-grid methods, were subsequently investigated by Jin and Kobayashi~\cite{JinKobayashi2019, JinKobayashi2021}.
For superlinear coefficients, truncated Euler, Milstein, and semi-implicit schemes were developed in \cite{DengLiu2020, LiLiuTang2025, LiuMaoTangWu2020,WuZuo2026}.
However, Euler approximations on a $D$-grid need not converge in general when the $\dd t$ drift depends on the state.
As observed by Jin and Kobayashi~\cite{JinKobayashi2021}, reducing the subordinator sampling step $\delta$ does not force the time increments on the $t$-axis to vanish.
Indeed, for the linear equation $\mathrm{d}X(t)=X(t)\dd t$, the terminal error of the $D$-grid update remains strictly bounded away from zero as $\delta\to0$, as established in \Cref{prop:Dgrid}.

To overcome the non-convergence of $D$-grid schemes, recent efforts have focused on developing numerical methods on equidistant time grids.
In particular, Zuo~\cite{Zuo2026} established parameter-dependent strong convergence rates of Euler-type methods for the autonomous time-changed SDE, which operates purely under the subordinator without a $\dd t$ drift.
Subsequently, Zuo and Guo~\cite{ZuoGuo2026} designed an equidistant truncated Euler--Maruyama method for nonautonomous superlinear equations with a state-dependent $\dd t$ drift, demonstrating that an equidistant time mesh successfully recovers strong convergence in the non-delayed setting.
A natural question then arises as to whether one can simply borrow these techniques for the delayed system \eqref{eq:model}.
However, this turns out to be far from straightforward because the presence of time delay introduces non-trivial difficulties in both the stability analysis and error propagation.
Specifically, delayed states enter all drift and diffusion coefficients, requiring the analysis to simultaneously control polynomial growth in both current and historical states, superlinear diffusion, and the interplay between the deterministic time increment and the random subordinator increment.

Beyond the challenges of time delay, the convergence analyses in existing equidistant studies~\cite{Zuo2026, ZuoGuo2026} are formulated using exact inverse-subordinator values at the time-grid points.
A standard approximation instead samples the underlying subordinator with step size $\delta>0$ and determines the inverse process through first passage times \cite{DengLiu2020, JinKobayashi2019, JumKobayashi2016, LiuMaoTangWu2020}.
Here, we include the resulting inverse-subordinator sampling error in the strong convergence analysis.
Assume that the subordinator increments can be sampled exactly.
Because the resulting discrete inverse subordinator takes values on the subordinator sampling grid and the driving Brownian motion is independent of the subordinator, increments of the time-changed Brownian motion can be simulated exactly by summing independent Gaussian increments, leaving the numerical scheme fully self-contained without requiring continuous Brownian interpolation.
Moreover, the pathwise discretization error incurred by such numerical inversion is known to achieve first-order accuracy with respect to the sampling resolution $\delta$~\cite{JinKobayashi2019, Magdziarz2009}.
For further developments in sampling inverse subordinators, including exact simulation, see \cite{GonzalezCazaresLinMijatovic2025,BiocicCedenoGironToaldo2026}.
We therefore develop a fully discrete ETEM method that simultaneously accounts for the equidistant time mesh and the discrete sampling of the subordinator.
Under polynomial growth in both current and delayed states, we rigorously quantify the joint error induced by both temporal discretization and subordinator sampling, identifying step-size choices that balance their contributions.

The main contributions are as follows.
\begin{enumerate}[label=(\roman*),leftmargin=*]
\item \textbf{Fully discrete strong convergence.}
We establish strong convergence rates for the fully discrete ETEM method under coupled one-sided and polynomial growth conditions.
The bounds quantify both time-stepping and inverse-subordinator sampling errors, with no prescribed relationship between $\Delta$ and $\delta$.
\item \textbf{Lower bound and sharpness.}
For a scalar multiplicative-noise equation with a stable subordinator, we prove the terminal strong lower bound $c(\Delta^{\alpha/2}+\delta^{1/2})$.
This establishes the sharpness of both polynomial exponents for the proposed scheme over the admissible equation class, up to logarithmic factors.
\end{enumerate}

The remainder of the paper is organized as follows.
\Cref{sec:preliminaries} provides the mathematical preliminaries, structural assumptions, and auxiliary estimates on inverse subordinators.
\Cref{sec:method-results} formulates the fully discrete ETEM scheme and states the strong convergence and lower-bound results.
\Cref{sec:convergence-analysis} develops the strong convergence analysis by comparing the partitions generated by the exact and approximate inverse subordinators.
\Cref{sec:experiments} illustrates convergence for scalar nonautonomous and two-dimensional superlinear delay equations, as well as delay equations with low-regularity time dependence.
Finally, \Cref{sec:conclusion} concludes the paper.
\needspace{7\baselineskip}
\section{Preliminaries}\label{sec:preliminaries}

In this section, we collect the necessary probabilistic foundations on inverse subordinators, introduce the governing structural assumptions on the coefficients and the initial history, and establish the global existence, uniqueness, and moment bounds of the exact solution.

Let $T>0$, $\tau>0$, and $d,m\in\mathbb N$.
Write $\R_+=[0,\infty)$ and denote the Euclidean and Frobenius norms by $|\cdot|$ and the Euclidean inner product by $\ip{\cdot}{\cdot}$.
The superscript $\top$ denotes transposition.
Let $(\Omega_B,\mathcal F_B,\PP_B)$ and $(\Omega_D,\mathcal F_D,\PP_D)$ carry an $m$-dimensional Brownian motion $B$ and a subordinator $D$, respectively.
We work on their product space with probability measure $\PP=\PP_B\otimes\PP_D$, so $B$ and $D$ are independent.
The notation $\EB$ denotes integration over the Brownian coordinate with the subordinator path fixed, and $\ED$ denotes integration over the subordinator coordinate.
Thus $\EE[\cdot]=\ED[\EB[\cdot]]$.

The subordinator $D=(D(u))_{u\geq0}$ has no killing and zero drift.
Its Laplace exponent $\psi$ and inverse $E$ are defined by
\begin{equation*}
    \EE[\exp\{-\lambda D(u)\}]=\exp\{-u\psi(\lambda)\},\qquad \lambda>0,\qquad E(t)=\inf\{u>0:D(u)>t\},\quad t\geq0.
\end{equation*}

Throughout this paper, the driftless subordinator $D$ is assumed to have an infinite L\'evy measure $\nu(0,\infty)=\infty$, ensuring that $D$ exhibits infinitely many jumps and strictly increasing paths almost surely.
We assume the lower growth condition $\psi(\lambda)\geq c_1\lambda^\alpha$ for all $\lambda\geq\lambda_0$, where $0<\alpha<1$, $c_1>0$, and $\lambda_0>0$ are fixed constants.
This is a weaker, one-sided power-growth requirement obtained by fixing the reference argument in the lower bound of the scaling condition in Kim, Song and Vondra\v{c}ek~\cite{KimSongVondracek2018}.
Earlier equidistant numerical schemes for TCSDEs were exclusively restricted to the standard $\alpha$-stable subordinator with $\psi(\lambda)=\lambda^\alpha$ \cite{Zuo2026, ZuoGuo2026}.
The lower growth condition adopted here substantially relaxes this restriction by accommodating a much broader family of driving subordinators without enforcing exact self-similarity.
Since the drift vanishes and $\psi(\lambda)\to\infty$ as $\lambda\to\infty$, we have $D(u)\to\infty$ almost surely, which guarantees that the inverse subordinator $E$ has continuous nondecreasing sample paths with $E(0)=0$.
Let $(\mathcal G_t)_{t\geq0}$ be the usual augmentation of the filtration generated at time $t$ by the full path of $D$ and the stopped Brownian path $(B(u\wedge E(t)))_{u\geq0}$.
Under $\PP_B$, the inverse subordinator $E$ and the measures $\dd t$ and $\dd E(t)$ are deterministic, while $B(E(t))$ is a continuous Gaussian martingale with quadratic covariation $E(t)I_m$.
We apply martingale estimates under $\PP_B$ and then average over $D$.
For each fixed choice of parameters, conditional assertions are understood on a $\PP_D$-full set of sample paths along which $D$ is strictly increasing and unbounded, so that $E$ is continuous, nondecreasing, and finite on $[0,T]$.

To collect the $\mathrm{d}t$ and $\mathrm{d}E$ contributions, set
\begin{equation*}
    A(t)=t+E(t).
\end{equation*}

The coefficients $h,f:[0,T]\times\R_+\times\R^d\times\R^d\to\R^d$ and $g:[0,T]\times\R_+\times\R^d\times\R^d\to\R^{d\times m}$ are Borel measurable.
For the presentation, take $T=K\tau$ with $K\in\mathbb{N}$.
A general finite terminal time can be treated by extending the coefficients constantly in their first time argument to the next multiple of $\tau$ and restricting the resulting solution.
All the assumptions below are preserved by this extension.

The spatial assumptions distinguish polynomial growth (\Cref{ass:polynomial}) from control of the current state (\Cref{ass:current}).
The former determines the higher moments needed for delayed terms (decoupling the historical memory), while the latter supports moment and error estimates on each delay interval (providing one-sided dissipative control to balance superlinear diffusion and prevent finite-time explosion).

\begin{assumption}\label{ass:polynomial}
There are constants $L_1>0$ and $\ell\geq1$ such that, for $t\in[0,T]$, $u\geq0$, and $x,x',y,y'\in\R^d$,
\begin{equation*}
    \sum_{\varphi\in\{h,f,g\}} |\varphi(t,u,x,y)-\varphi(t,u,x',y')| \leq L_1(1+|x|^\ell+|x'|^\ell+|y|^\ell+|y'|^\ell) (|x-x'|+|y-y'|),
\end{equation*}
and $|h(t,u,0,0)|+|f(t,u,0,0)|+|g(t,u,0,0)|\leq L_1$.
\end{assumption}

In particular, \Cref{ass:polynomial} gives the uniform growth estimate
\begin{equation}\label{eq:growth}
    \sum_{\varphi\in\{h,f,g\}}|\varphi(t,u,x,y)| \leq C(1+|x|^{\ell+1}+|y|^{\ell+1}).
\end{equation}

\begin{assumption}\label{ass:current}
There is a constant $K_1>0$ such that, for $t\in[0,T]$, $u\geq0$, and $x,x',y\in\R^d$,
\begin{equation*}
    \ip{x-x'}{h(t,u,x,y)-h(t,u,x',y)}\leq K_1|x-x'|^2.
\end{equation*}
For every $p\geq2$, there is a constant $K_2>0$, allowed to depend on $p$, such that
\begin{equation*}
    \ip{x-x'}{f(t,u,x,y)-f(t,u,x',y)}+\frac{p-1}{2}|g(t,u,x,y)-g(t,u,x',y)|^2\leq K_2|x-x'|^2.
\end{equation*}
Both constants are independent of $t,u$ and $y$.
\end{assumption}

\begin{assumption}\label{ass:time}
There are constants $L_2>0$ and $\beta_1,\beta_2\in(0,1]$ such that, for $s,t\in[0,T]$, $u,v\geq0$, and $x,y\in\R^d$,
\begin{equation*}
    \sum_{\varphi\in\{h,f,g\}} |\varphi(t,u,x,y)-\varphi(s,v,x,y)| \leq L_2(1+|x|^{\ell+1}+|y|^{\ell+1}) (|t-s|^{\beta_1}+|u-v|^{\beta_2}).
\end{equation*}
\end{assumption}

\begin{assumption}\label{ass:history}
The initial history $\xi:[-\tau,0]\to\R^d$ is deterministic and bounded, and satisfies
\begin{equation*}
    |\xi(t)-\xi(s)|\leq L_2|t-s|^\gamma, \qquad s,t\in[-\tau,0],
\end{equation*}
for some $\gamma\in(0,1]$.
\end{assumption}
The H\"older condition on the initial history in \Cref{ass:history} is standard in the numerical analysis of stochastic delay differential equations (see, e.g., \cite{MaoSabanis2003, WuMao2008}).

Throughout the paper, $C$ denotes a generic positive constant that may depend on the fixed exponents appearing in the estimate, $T,\tau,\alpha,{c_1,\lambda_0}$, the assumption constants, and the initial history.
It may change between occurrences and is independent of the mesh sizes $\Delta,\delta$ and of the subordinator path being fixed.
To average the conditional error bounds over $D$, we need moment estimates for inverse subordinator increments with exponential weights.
The following lemma establishes the moment and exponential integrability estimates of the inverse subordinator (see \cite{BertoinCaballero1995, JinKobayashi2021, JumKobayashi2016, Zuo2026, ZuoGuo2026}).

\begin{lemma}\label{lem:inverse}
For every $c>0$ and $q>0$, $\EE[\exp\{cE(T)\}]<\infty$ and
\begin{equation*}
    \EE[(E(t+r)-E(t))^q]\leq Cr^{\alpha q}, \qquad t\geq0,\quad 0<r\leq2T.
\end{equation*}
For $0<\Delta\leq1\wedge T$, define
\begin{equation*}
    \Xi_\Delta=\sup_{0\leq s\leq T-\Delta}\{E(s+\Delta)-E(s)\}.
\end{equation*}
Then
\begin{equation*}
    \EE[\Xi_\Delta^q]+\EE[e^{cE(T)}\Xi_\Delta^q]\leq C\Delta^{\alpha q}(1+|\log\Delta|)^{(1-\alpha)q}.
\end{equation*}
In particular, both terms are bounded by $C\Delta^{\rho q}$ for every fixed $0<\rho<\alpha$.
\end{lemma}
\begin{proof}
Since $E(t)$ is a stopping time with $D(E(t))\geq t$, the event $\{E(t+r)-E(t)>s\}$ implies $D(E(t)+s)-D(E(t))\leq r$.
By the strong Markov property and the exponential Markov inequality, for every $\lambda>0$,
\begin{equation*}
	\PP(E(t+r)-E(t)>s)\leq\PP(D(s)\leq r)\leq\exp\{\lambda r-s\psi(\lambda)\}.
\end{equation*}
Using the lower growth condition and choosing $\lambda=(\alpha c_1s/r)^{1/(1-\alpha)}$ when $\lambda\geq\lambda_0$, we obtain
\begin{equation}\label{eq:inverse-increment-tail}
	\PP(E(t+r)-E(t)>s)\leq C_1\exp\{-c_2(s/r^\alpha)^{1/(1-\alpha)}\},\qquad t\geq0,\quad s>0,\quad 0<r\leq2T.
\end{equation}
When this choice of $\lambda$ is below $\lambda_0$, the ratio $s/r^\alpha$ is uniformly bounded for $0<r\leq2T$, so the same estimate holds after enlarging $C_1\geq1$.
Integrating this tail bound gives $\EE[e^{cE(T)}]<\infty$ since $1/(1-\alpha)>1$, as well as $\EE[(E(t+r)-E(t))^q]\leq Cr^{\alpha q}$ via the substitution $v=s/r^\alpha$.

For the uniform modulus, put $M_1=\lceil T/\Delta\rceil$.
If $j\Delta\leq u<(j+1)\Delta$, the monotonicity of $E$ gives $E(u+\Delta)-E(u)\leq E((j+2)\Delta)-E(j\Delta)$, whence $\{\Xi_\Delta>s\}\subseteq\bigcup_{j=0}^{M_1-1}\{E((j+2)\Delta)-E(j\Delta)>s\}$.
Applying the union bound and \eqref{eq:inverse-increment-tail} with $r=2\Delta$ yields
\begin{equation*}
    \PP(\Xi_\Delta>s)\leq\min\{1,C_1M_1\exp[-c_2(s/(2\Delta)^\alpha)^{1/(1-\alpha)}]\}.
\end{equation*}
Setting $s_*=(2\Delta)^\alpha[2\log(C_1M_1)/c_2]^{1-\alpha}$, the second term in the minimum does not exceed $\exp[-\frac12c_2(s/(2\Delta)^\alpha)^{1/(1-\alpha)}]$ for all $s\geq s_*$.

Splitting the moment integral $\EE[\Xi_\Delta^q]=q\int_0^\infty s^{q-1}\PP(\Xi_\Delta>s)\dd s$ at $s_*$, we obtain
\begin{align*}
    \EE[\Xi_\Delta^q]
    &\leq q\int_0^{s_*} s^{q-1}\dd s+q\int_{s_*}^\infty s^{q-1}\exp[-\tfrac12c_2(s/(2\Delta)^\alpha)^{1/(1-\alpha)}]\dd s\\
    &\leq s_*^q+C\Delta^{\alpha q}\\
    &\leq C\Delta^{\alpha q}(1+|\log\Delta|)^{(1-\alpha)q},
\end{align*}
where the substitution $v=s/(2\Delta)^\alpha$ bounds the tail integral, and $M_1\leq(1+T)/\Delta$ bounds $s_*^q$ in the last step.
Finally, Cauchy--Schwarz gives $\EE[e^{cE(T)}\Xi_\Delta^q]\leq(\EE[e^{2cE(T)}]\EE[\Xi_\Delta^{2q}])^{1/2}\leq C\Delta^{\alpha q}(1+|\log\Delta|)^{(1-\alpha)q}\leq C\Delta^{\rho q}$ for every fixed $\rho<\alpha$.
\end{proof}
\begin{remark}\label{rem:inverse-comparison}
For a standard inverse $\alpha$-stable subordinator, Gajda and Magdziarz~\cite{GajdaMagdziarz2014} show that, for fixed $t>0$ and every integer $n\geq1$, $\EE[(E(t+r)-E(t))^n]$ is of order $r^{1+(n-1)\alpha}$ as $r\to0$, whereas $\EE[E(r)^n]=n!r^{n\alpha}/\Gamma(1+n\alpha)$ at the origin.
For the general subordinator class considered here, \Cref{lem:inverse} gives $\EE[(E(t+r)-E(t))^q]\leq Cr^{\alpha q}$ uniformly over $t\geq0$.
The strong error analysis below uses the exponentially weighted maximum-increment estimate in that lemma.

For the inverse stable subordinator, Zuo~\cite{Zuo2026} derives a maximum-increment moment bound using a uniform modulus of continuity and self-similarity.
By combining an increment tail bound with a finite covering argument, \Cref{lem:inverse} establishes the same logarithmic moment order, $\Delta^{\alpha q}(1+|\log\Delta|)^{(1-\alpha)q}$, appearing in that proof for the broader class of subordinators satisfying $\psi(\lambda)\geq c_1\lambda^\alpha$ for $\lambda\geq\lambda_0$.
In particular, taking $\psi(\lambda)=\lambda^\alpha$ recovers the stable subordinator estimate.
\end{remark}

The next lemma establishes a logarithmic modulus of continuity for stochastic integrals with random bounds, which will be utilized to control the endpoints of the continuous and discrete inverse-subordinator Brownian partitions.
\begin{lemma}\label{lem:martingale-modulus}
Let $L\geq1$, $0<\delta\leq1$, and $p\geq2$.
Let $G$ be a predictable matrix-valued process on $[0,L]$, put $M(u)=\int_0^uG(v)\dd B(v)$, and suppose $S=\sup_{0\leq v\leq L}|G(v)|$ belongs to $L^{2p}(\PP_B)$.
Then
\begin{equation*}
    \EB\left[\sup_{\substack{0\leq u<v\leq L\\v-u\leq\delta}}|M(v)-M(u)|^p\right]\leq C(\EB[S^{2p}])^{1/2}[\delta\log(e+L/\delta)]^{p/2},
\end{equation*}
where $C$ depends only on $p$ and the dimensions.
\end{lemma}
The proof is given in \Cref{app:martingale-modulus}.

\begin{example}\label{ex:clocks}
The following driving subordinators satisfy the lower growth condition $\psi(\lambda)\geq c_1\lambda^\alpha$, demonstrating that our framework covers physically realistic processes beyond standard $\alpha$-stable subordinators.
The tempered stable exponent $\psi(\lambda)=(\lambda+\mu)^\alpha-\mu^\alpha$ with $\mu>0$ obeys $\psi(\lambda)\geq\lambda^\alpha/2$ once $\lambda^\alpha\geq2\mu^\alpha$.
Its L\'evy density is $\alpha e^{-\mu z}z^{-1-\alpha}/\Gamma(1-\alpha)$, where the exponential tempering tames extremely large jumps while retaining $\alpha$-subdiffusive behavior at high frequencies.
For approximate simulation of tempered stable distributions, see \cite{GrabchakSaba2025}.
A finite sum $\psi(\lambda)=\sum_{k=1}^M w_k\lambda^{\alpha_k}$ with $w_k>0$ and $0<\alpha_k<1$ satisfies the condition with $\alpha=\max_k\alpha_k$ by retaining a corresponding dominant positive summand.
A stable L\'evy density restricted to $0<z\leq R$ for some cutoff threshold $R>0$ also satisfies the condition, since
\begin{align*}
    \psi(\lambda)&\geq\lambda^\alpha-\frac{\alpha}{\Gamma(1-\alpha)}\int_R^\infty z^{-1-\alpha}\dd z\\
    &\geq\tfrac12\lambda^\alpha
\end{align*}
for all sufficiently large $\lambda$.
This truncated subordinator produces sample paths with bounded jumps, which is widely utilized in financial engineering to ensure finite moments of all orders.
In contrast, Gamma subordinators have logarithmic Laplace growth and do not meet this lower power-law condition.
\end{example}
Compared with existing equidistant studies \cite{Zuo2026, ZuoGuo2026} where the subordinator is strictly restricted to the pure $\alpha$-stable case $\psi(\lambda)=\lambda^\alpha$, the lower-growth condition supplies a robust uniform modulus for the $t$-mesh analysis across all the aforementioned subordinator families.

The delay must be estimated with respect to its original integration measure.
For any bounded Borel function $V:\R^d\to\R$ and $\tau\leq t\leq T$,
\begin{equation*}
    \int_\tau^t V(X(s-\tau))\dd E(s)=\int_0^{t-\tau}V(X(r))\dd E(r+\tau).
\end{equation*}
The measure on the right cannot in general be replaced by $\dd E(r)$.
On each delay interval, the delayed state belongs to the already treated history.
The next estimate bounds the delayed integral using higher moments of the history while retaining the measure $\mathrm{d}A(s)$ at the current integration time.

\begin{lemma}\label{lem:delay}
Fix a subordinator path, an integer $1\leq k\leq K$, and $p>0$.
Let $Z$ be a process on $[-\tau,(k-1)\tau]$, and let $W$ be a nonnegative measurable process on $[(k-1)\tau,k\tau]$.
If the moments on the right-hand side below are finite, then
\begin{align*}
    \EB\left[\int_{(k-1)\tau}^{k\tau} W(s)|Z(s-\tau)|^p\dd A(s)\right] &\leq (A(k\tau)-A((k-1)\tau)) \sup_{(k-1)\tau\leq s\leq k\tau} \bigl(\EB[W(s)^2]\bigr)^{1/2}\\*
    &\quad\times \left(\EB\left[\sup_{-\tau\leq r\leq(k-1)\tau} |Z(r)|^{2p}\right]\right)^{1/2}.
\end{align*}
For $W=1$, the last two factors can instead be replaced by $\EB[\sup_{-\tau\leq r\leq(k-1)\tau}|Z(r)|^p]$.
\end{lemma}

\begin{proof}
Under $\PP_B$, $A$ is deterministic and $s-\tau\leq(k-1)\tau$ on the interval of integration.
Tonelli's theorem and Cauchy--Schwarz give the first assertion.
For $W=1$, bound the delayed value directly by the history supremum.
\end{proof}

To resolve integral inequalities driven by the continuous nondecreasing clock $A(t)=t+E(t)$, we also recall the generalized Gronwall inequality on Stieltjes measures \cite{jacod2013limit,long2024ulam}.

\begin{lemma}\label{lem:stieltjes-gronwall}
Let $A$ be a continuous nondecreasing function on $[0,T]$ with $A(0)=0$.
Suppose that a nonnegative locally bounded Borel function $z$ satisfies
\begin{equation*}
    z(t) \leq a + c\int_0^t z(s)\dd A(s), \qquad 0\leq t\leq T,
\end{equation*}
where $a\geq0$ and $c\geq0$ are constants.
Then
\begin{equation*}
    z(t) \leq a e^{cA(t)}, \qquad 0\leq t\leq T.
\end{equation*}
\end{lemma}

We first apply these estimates to the exact solution.
The resulting moments will also control the polynomial weights in the numerical error analysis.
A strong solution is a continuous process adapted to $(\mathcal G_t)_{t\geq0}$, with initial history $\xi$, which satisfies the integral form of \eqref{eq:model} for the given $B$ and $D$.

\Cref{prop:exact} ensures that, despite superlinear current-state growth and polynomial delayed growth, the TCSDDE \eqref{eq:model} admits a unique nonexplosive global strong solution whose moments of all orders remain uniformly bounded both conditionally under $\PP_B$ and unconditionally under $\PP$.
\begin{proposition}\label{prop:exact}
Under \Cref{ass:polynomial,ass:current,ass:time,ass:history}, \eqref{eq:model} admits a unique nonexplosive continuous strong solution satisfying, for every $p\geq2$,
\begin{equation*}
    \EB\left[\sup_{-\tau\leq t\leq T}|X(t)|^p\right] \leq Ce^{CA(T)}, \qquad \EE\left[\sup_{-\tau\leq t\leq T}|X(t)|^p\right]\leq C.
\end{equation*}
\end{proposition}

The proof is given in \Cref{app:exact-moments}.

\section{The ETEM Scheme}\label{sec:method}\label{sec:method-results}\label{sec:fully-discrete}
This section formulates the ETEM scheme for the delay system \eqref{eq:model} and establishes its convergence properties.

\subsection{Failure of the subordinator grid for state-dependent drift}\label{sec:subordinator-grid-failure}

In existing numerical studies for TCSDEs \cite{DengLiu2020, JinKobayashi2019, JinKobayashi2021, JumKobayashi2016, LiLiuTang2025, LiuMaoTangWu2020, WuZuo2026}, a widely adopted discretization strategy is the $D$-grid method, which generates discrete time nodes by sampling the subordinator on an equidistant partition with step size $\delta>0$,
\begin{equation*}
    \tau_i=D(i\delta),\qquad 0\leq i<N_\delta,\qquad \tau_{N_\delta}=T,
\end{equation*}
where $N_\delta=\min\{n\geq1:D(n\delta)\geq T\}$.
Setting $\Delta\tau_i=\tau_{i+1}-\tau_i$, this approach reduces the inverse subordinator increment to the deterministic constant $\delta$, which simplifies Brownian integration.
However, as observed by Jin and Kobayashi~\cite{JinKobayashi2021}, reducing the sampling step $\delta$ does not force the time increments $\Delta\tau_i$ to vanish.
\Cref{prop:Dgrid} provides a rigorous counterexample confirming that this represents a fundamental barrier.
For the state-dependent $\mathrm{d}t$ drift in the following example, the frozen $D$-grid Euler update has a terminal $L^p$ error bounded away from zero as $\delta\to0$.

\begin{proposition}\label{prop:Dgrid}
For $\mathrm{d}X(t)=X(t)\dd t$ with $X(0)=1$, let
\begin{equation*}
    Y_0=1,\qquad Y_{i+1}=(1+\Delta\tau_i)Y_i.
\end{equation*}
For every $p>0$, there is a constant $c_0>0$, independent of $\delta$, such that
\begin{equation*}
    \EE[|X(T)-Y_{N_\delta}|^p]\geq c_0.
\end{equation*}
\end{proposition}
\begin{proof}
Recall that $\nu$ denotes the L\'evy measure of $D$.
Since $\nu(0,\infty)=\infty$, we can choose $0<c\leq T/2$ such that $\nu((c,\infty))>0$.
By the L\'evy--It\^o decomposition, we separate $D=D_{\leq c}+D_{>c}$ into jumps of size at most $c$ and jumps exceeding $c$.
These two processes are mutually independent.
The large-jump component $D_{>c}$ is a compound Poisson process with jump rate $\lambda_c=\nu((c,\infty))>0$.
Let $S:=\inf\{u>0:\Delta D(u)>c\}$ denote the arrival time of its first jump.
Then $S$ follows an exponential distribution with parameter $\lambda_c$, satisfying
\begin{equation}\label{eq:large-jump-time}
    \PP(S\leq u_0)=1-e^{-\lambda_cu_0}>0,\qquad u_0>0.
\end{equation}
For the small-jump component $D_{\leq c}$, its mean satisfies
\begin{equation*}
    \EE[D_{\leq c}(u)]=u\int_{(0,c]}z\nu(\dd z).
\end{equation*}
Since the integral is finite, this expectation vanishes as $u\to0$.
We can therefore choose $u_0>0$ sufficiently small such that $\EE[D_{\leq c}(u_0)]\leq T/8$.
Applying Markov's inequality to $D_{\leq c}(u_0)$ yields
\begin{equation*}
    \PP(D_{\leq c}(u_0)>T/4)\leq\frac{\EE[D_{\leq c}(u_0)]}{T/4}\leq\tfrac12,
\end{equation*}
which implies $\PP(D_{\leq c}(u_0)\leq T/4)\geq\frac12$.
Now define the favorable event
\begin{equation}\label{eq:favorable-event}
    \Omega_0:=\bigl\{S\leq u_0,\ D_{\leq c}(u_0)\leq T/4\bigr\}.
\end{equation}
Because $D_{\leq c}$ and $D_{>c}$ are independent, $\Omega_0$ occurs with strictly positive probability,
\begin{equation}\label{eq:event-prob}
    \PP(\Omega_0)=\PP(S\leq u_0)\,\PP\bigl(D_{\leq c}(u_0)\leq T/4\bigr)\geq\tfrac12(1-e^{-\lambda_cu_0})>0.
\end{equation}
We now analyze the trajectory of $D$ on the event $\Omega_0$.
By the definition of $S$, no jumps strictly larger than $c$ occur on the time interval $[0,S)$, which means $D_{>c}(S-)=0$.
Consequently, the subordinator immediately prior to this first large jump satisfies
\begin{equation*}
    D(S-)=D_{\leq c}(S-)+D_{>c}(S-)=D_{\leq c}(S-).
\end{equation*}
Since the small-jump path $u\mapsto D_{\leq c}(u)$ is nondecreasing and $S\leq u_0$ on $\Omega_0$, we have
\begin{equation*}
    D(S-)\leq D_{\leq c}(u_0)\leq T/4.
\end{equation*}
On $\Omega_0$, since $\Delta D(S)>c$, $D(S-)\leq T/4$, and $c\leq T/2$, we have
\begin{equation*}
	\min\{D(S),T\}-D(S-)=\min\{\Delta D(S),T-D(S-)\}\geq c.
\end{equation*}
Hence the sampled grid interval containing this jump has length at least $c$, even when truncated at $T$.
Since the function $r-\log(1+r)$ is nondecreasing on $[0,\infty)$, on this same event we have
\begin{align}\label{eq:Y-terminal-bound}
    Y_{N_\delta}&=\prod_{i=0}^{N_\delta-1}(1+\Delta\tau_i) \nonumber\\
    &=\exp\left\{T-\sum_{i=0}^{N_\delta-1}[\Delta\tau_i-\log(1+\Delta\tau_i)]\right\} \nonumber\\
    &\leq e^T\exp\{-c+\log(1+c)\} \nonumber\\
    &=(1+c)e^{T-c}.
\end{align}
Because $1+c<e^c$ for every $c>0$, the numerical solution strictly underestimates the exact solution $X(T)=e^T$ with pointwise discrepancy
\begin{equation*}
    |X(T)-Y_{N_\delta}|\geq e^T\bigl(1-(1+c)e^{-c}\bigr)>0.
\end{equation*}
Restricting the $L^p$ error to the favorable event $\Omega_0$ and applying \eqref{eq:event-prob}, we obtain
\begin{align}\label{eq:terminal-expectation-lower}
    \EE[|X(T)-Y_{N_\delta}|^p]
    &\geq\EE\bigl[|X(T)-Y_{N_\delta}|^p\mathbf{1}_{\Omega_0}\bigr] \nonumber\\
    &\geq\tfrac12(1-e^{-\lambda_cu_0})e^{pT}\bigl(1-(1+c)e^{-c}\bigr)^p\nonumber\\
    &=:c_0>0.
\end{align}
Since $c_0$ depends only on $T$, $c$, $u_0$, and $p$, and is independent of $\delta$, this completes the proof.
\end{proof}

\subsection{Construction of the ETEM scheme}\label{sec:fd-scheme}
To control the increments of the $\mathrm{d}t$ drift, following Zuo and Guo~\cite{ZuoGuo2026}, we use an equidistant partition in $t$.
Let $J\in\mathbb N$, $\Delta=\tau/J\leq1\wedge T$, $N=KJ$, and $t_n=n\Delta$.
Then $t_n-\tau=t_{n-J}$, so every delayed grid argument is an existing node.
Define $\fl{s}=t_n$ and $n(s)=n$ for $s\in[t_n,t_{n+1})$, including the negative grid intervals.

Let $0<\delta\leq1$. Following \cite{JumKobayashi2016,MagdziarzOrzelWeron2011}, we define the discrete inverse subordinator by
\begin{equation}\label{eq:lower-clock}
    E^\delta(t)=\delta\bigl(\min\{j\geq1:D(j\delta)>t\}-1\bigr).
\end{equation}
Set $E_n=E(t_n)$, $\Delta E_n=E_{n+1}-E_n$, $E_n^\delta=E^\delta(t_n)$, $\Delta E_n^\delta=E_{n+1}^\delta-E_n^\delta$, and $\Delta B_n=B(E_{n+1}^\delta)-B(E_n^\delta)$.
By the strict increase of $D$, the approximation satisfies the uniform bounds (see, e.g., \cite{JumKobayashi2016,DengLiu2020})
\begin{equation}\label{eq:clock-bracket}
    0\leq E(t)-E^\delta(t)\leq\delta,\qquad |\Delta E_n^\delta-\Delta E_n|\leq\delta.
\end{equation}
Indeed, if $D(j\delta)\leq t<D((j+1)\delta)$, then $j\delta\leq E(t)\leq(j+1)\delta$ and $E^\delta(t)=j\delta$.
This also gives the endpoint convention when $t=D(j\delta)$.
For the time-changed Brownian motion, the It\^o isometry gives
\begin{align*}
    \EB[|B(E(t))-B(E^\delta(t))|^2]&=m(E(t)-E^\delta(t))\\
    &\leq m\delta.
\end{align*}
Thus an inverse subordinator discrepancy of size $\delta$ gives a pointwise root-mean-square (RMS) bound of order $\delta^{1/2}$.
Controlling the maximum in time through \Cref{lem:martingale-modulus} introduces the logarithmic factor used in the strong error estimate.

Let $X_n$ denote the ETEM approximation to $X(t_n)$, with its dependence on $\Delta,\delta$ left implicit, and set $X_j=\xi(t_j)$ for $-J\leq j\leq0$.
Write $\varphi(X_n)=\varphi(t_n,E_n^\delta,X_n,X_{n-J})$ for $\varphi\in\{h,f,g\}$, and define the common denominator
\begin{align}
    \mathcal D_n&=1+\sqrt{\Delta E_n^\delta}(|f(X_n)|+|g(X_n)|^2),\notag\\
    h_{\Delta,n}&=\frac{h(X_n)}{1+\sqrt\Delta\,|h(X_n)|},\notag\\
    f_n&=\frac{f(X_n)}{\mathcal D_n},\qquad g_n=\frac{g(X_n)}{\sqrt{\mathcal D_n}}.\label{eq:fd-coefficients}
\end{align}
The fully discrete ETEM scheme is
\begin{equation}\label{eq:fd-scheme}
    X_{n+1}=X_n+h_{\Delta,n}\Delta+\Delta E_n^\delta f_n +g_n\Delta B_n,\qquad 0\leq n<N.
\end{equation}
Under $\PP_B$, the Brownian increment $\Delta B_n$ is a centered Gaussian vector with covariance $\Delta E_n^\delta I_m$, independent of the previous states $X_0,\dots,X_n$, which also covers the degenerate case $\Delta E_n^\delta=0$.
When $\Delta E_n^\delta=0$, only the $\mathrm{d}t$ drift contributes.
In particular, the scheme remains well defined when several $t$-grid nodes have the same approximate inverse subordinator value.

Let $\widetilde X(s)=X_{n(s)}$ for $s\in[-\tau,T)$ and $\widetilde X(T)=X_N$.
The following algebraic estimates will be used in both the moment and residual bounds.
The common denominator preserves the single-state dissipativity in \eqref{eq:energy}
\begin{equation*}
    \ip{X_n}{f_n}+p|g_n|^2\leq C(1+|X_n|^2+|X_{n-J}|^{2(\ell+1)}).
\end{equation*}
The two drift increments satisfy
\begin{align}
    |\Delta h_{\Delta,n}+\Delta E_n^\delta f_n|&\leq\frac{\Delta|h(X_n)|}{1+\sqrt\Delta|h(X_n)|}+\frac{\Delta E_n^\delta|f(X_n)|}{1+\sqrt{\Delta E_n^\delta}|f(X_n)|}\notag\\
    &\leq\sqrt\Delta+\sqrt{\Delta E_n^\delta}\leq\sqrt{2(\Delta+\Delta E_n^\delta)}.\label{eq:driftstep}
\end{align}
The taming biases obey
\begin{align}
    |h(X_n)-h_{\Delta,n}|&\leq\sqrt\Delta\,|h(X_n)|^2,\notag\\
    |f(X_n)-f_n|&\leq\sqrt{\Delta E_n^\delta}|f(X_n)|\bigl(|f(X_n)|+|g(X_n)|^2\bigr),\notag\\
    |g(X_n)-g_n|&\leq\tfrac12\sqrt{\Delta E_n^\delta}|g(X_n)|\bigl(|f(X_n)|+|g(X_n)|^2\bigr).\label{eq:tamebias}
\end{align}
The last inequality follows from $1-(1+r)^{-1/2}\leq r/2$ for $r\geq0$.
Moreover,
\begin{equation}\label{eq:diffusionstep}
    \Delta E_n^\delta|g_n|^2\leq\sqrt{\Delta E_n^\delta}.
\end{equation}

\subsection{Moment bounds and strong convergence}\label{sec:main-results}

With the ETEM scheme in place, we formulate its uniform moment bounds and two-parameter strong convergence rates.
The first theorem establishes moment bounds over all mesh nodes on $[0,T]$, with constants independent of $\Delta$ and $\delta$.

\begin{theorem}\label{thm:moments}
Under \Cref{ass:polynomial,ass:current,ass:time,ass:history}, for every $p\geq2$,
\begin{equation*}
    \EB\left[\max_{-J\leq n\leq N}|X_n|^p\right]\leq Ce^{CA(T)}.
\end{equation*}
Extend $E_n^\delta=0$ for $-J\leq n<0$.
For $-J\leq i\leq j\leq N$,
\begin{equation*}
    \EB[|X_j-X_i|^p]\leq Ce^{CA(T)}\bigl[|t_j-t_i|^p+|t_j-t_i|^{\gamma p}+(E_j^\delta-E_i^\delta)^{p/2}\bigr].
\end{equation*}
If $i\geq0$, the term $|t_j-t_i|^{\gamma p}$ can be omitted.
Moreover,
\begin{equation*}
    \EE\left[\max_{-J\leq n\leq N}|X_n|^p\right]\leq C.
\end{equation*}
\end{theorem}

The detailed proof is given in \Cref{sec:proof-moments}.

To simplify the notation, we define the following error scales
\begin{equation}\label{eq:clock-error-scale}
    \omega_\Delta=\Delta^\alpha(1+|\log\Delta|)^{1-\alpha},\qquad \varepsilon_\delta=\delta^\gamma+\delta^{\beta_1}+\delta^{\beta_2}+\sqrt{\delta(1+|\log\delta|)}.
\end{equation}

\begin{theorem}\label{thm:fully-discrete}
Under \Cref{ass:polynomial,ass:current,ass:time,ass:history}, for every $p\geq2$, the ETEM scheme \eqref{eq:fd-scheme} satisfies
\begin{equation*}
    \left(\EE[\max_{0\leq n\leq N}|X(t_n)-X_n|^p]\right)^{1/p}\leq C\left[\Delta^\gamma+\Delta^{\beta_1}+\omega_\Delta^{\beta_2}+\omega_\Delta^{1/2}+\varepsilon_\delta\right].
\end{equation*}
In particular, for every fixed $0<\rho<\alpha$ and $0<\theta<1/2$,
\begin{equation*}
    \left(\EE[\max_{0\leq n\leq N}|X(t_n)-X_n|^p]\right)^{1/p} \leq C(\Delta^r+\delta^\eta),
\end{equation*}
where
\begin{equation*}
    r=\min\{\gamma,\beta_1,\rho\beta_2,\rho/2\},\qquad \eta=\min\{\gamma,\beta_1,\beta_2,\theta\}.
\end{equation*}
\end{theorem}
The detailed proof is given in  \Cref{sec:proof-convergence}.

\begin{remark}\label{rem:balanced-clock}
\Cref{thm:fully-discrete} also guides the choice of the subordinator sampling step.
For Lipschitz time coefficients and Lipschitz initial history, the two contributions to the error bound are $C\omega_\Delta^{1/2}$ and $C\sqrt{\delta(1+|\log\delta|)}$.
Choosing $\delta\asymp\Delta^\alpha(1+|\log\Delta|)^{-\alpha}$ balances these contributions and gives the nodal strong error bound $C\omega_\Delta^{1/2}$.
Since $\alpha<1$, this choice allows $\delta$ to be larger than $\Delta$ for sufficiently small $\Delta$, so the subordinator can be sampled on a coarser mesh than the time grid.
\end{remark}

\subsection{Sharpness and matching lower bound}\label{sec:main-lower-bound}

A fundamental question in numerical SDE analysis is whether the established strong convergence rates are optimal.
The following theorem establishes a matching terminal root-mean-square lower bound of order $c(\Delta^{\alpha/2}+\delta^{1/2})$ for an admissible scalar TCSDDE with multiplicative noise driven by an $\alpha$-stable subordinator.

\begin{theorem}\label{thm:lower-bound}
Let $D$ have Laplace exponent $\lambda^\alpha$, $0<\alpha<1$, and let $T>0$ be fixed.
In \eqref{eq:model}, take $h=f=0$, $g(t,u,x,y)=x$, and $\xi=1$ on $[-\tau,0]$.
For the ETEM scheme \eqref{eq:fd-scheme} on the aligned meshes $\Delta=\tau/J$, $N=T/\Delta$, there are $c>0$, $\Delta_0>0$, and $\delta_0>0$, depending only on $\alpha,T$, such that
\begin{equation*}
    \left(\EE[|X(T)-X_N|^2]\right)^{1/2}\geq c(\Delta^{\alpha/2}+\delta^{1/2}),\qquad 0<\Delta=T/N\leq\Delta_0,\quad 0<\delta\leq\delta_0.
\end{equation*}
The two steps need not satisfy a relation.
Since this equation ignores the delayed state, the same recursion and lower bound also hold for arbitrary $\Delta=T/N$.
\end{theorem}

The terminal lower bound establishes the sharpness of both polynomial exponents $\Delta^{\alpha/2}$ and $\delta^{1/2}$ for this Euler scheme over the admissible equation class, up to logarithmic factors.
Each exponent is constrained separately, since the other step can be taken arbitrarily smaller.
The complete proof of \Cref{thm:lower-bound} is given in \Cref{sec:proof-lower-bound}.

\section{Strong Convergence Analysis}\label{sec:convergence-analysis}

This section presents the proofs of the moment bounds, strong convergence, and lower bound stated in \Cref{sec:method-results}.
The strong convergence proof uses an auxiliary continuous process and a comparison of the partitions generated by the exact and approximate inverse subordinators.

\subsection{Proof of the moment bounds}\label{sec:proof-moments}

We first establish the uniform moment bounds of the ETEM method asserted in \Cref{thm:moments}.
The proof combines a one-step Lyapunov estimate with recursive moment estimates on the delay intervals and a Brownian integral representation.

\begin{lemma}\label{lem:onestep}
Under \Cref{ass:polynomial,ass:current}, for every $p\geq2$, the ETEM solution \eqref{eq:fd-scheme} satisfies
\begin{align*}
    \EB[(1+|X_{n+1}|^2)^{p/2}] &\leq \EB[(1+|X_n|^2)^{p/2}] +C\bigl((\Delta+\Delta E_n^\delta)+(\Delta+\Delta E_n^\delta)^{p/2}\bigr)\\
    &\quad\times \bigl[1+\EB[(1+|X_n|^2)^{p/2}]+\EB[|X_{n-J}|^{(\ell+1)p}]\bigr].
\end{align*}
\end{lemma}

\begin{proof}
We use the coefficient notation in \eqref{eq:fd-coefficients}.
The discrete ETEM update \eqref{eq:fd-scheme} then reads
\begin{equation*}
    X_{n+1} = X_n + \Delta h_{\Delta,n} + \Delta E_n^\delta f_n + g_n\Delta B_n.
\end{equation*}

We first consider the case $\Delta+\Delta E_n^\delta \leq 1$.
Let $V(x)=(1+|x|^2)^{p/2}$ and write the increment $\Delta X_n:=X_{n+1}-X_n$.
Expanding $V(X_{n+1})=V(X_n+\Delta X_n)$ at $X_n$ to third order yields
\begin{equation*}
    V(X_{n+1}) = V(X_n) + \nabla V(X_n)\cdot\Delta X_n + \tfrac{1}{2}\Delta X_n^\top \nabla^2 V(X_n)\Delta X_n + \tfrac{1}{6}\nabla^3 V(X_n)[\Delta X_n]^3 + R_4,
\end{equation*}
where $R_4$ is the fourth-order Taylor remainder.
Since $\Delta B_n$ has mean zero and is independent of $(X_n,X_{n-J})$, all terms containing an odd number of copies of $\Delta B_n$ vanish under $\EB[\cdot]$.
For the leading first-order drift and second-order diffusion terms, using \eqref{eq:energy} and Young's inequality, we obtain
\begin{align}
    &\quad \EB\bigl[\nabla V(X_n)\cdot(\Delta h_{\Delta,n}+\Delta E_n^\delta f_n) + \tfrac{1}{2}(g_n\Delta B_n)^\top \nabla^2 V(X_n)(g_n\Delta B_n)\bigr]\notag\\
    &\leq p\,\EB\left[(1+|X_n|^2)^{p/2-1}\left(\ip{X_n}{\Delta h_{\Delta,n}+\Delta E_n^\delta f_n} + \frac{p-1}{2}\Delta E_n^\delta|g_n|^2\right)\right]\notag\\
    &\leq C(\Delta+\Delta E_n^\delta)\,\EB\bigl[(1+|X_n|^2)^{p/2-1}\bigl(1+|X_n|^2+|X_{n-J}|^{2(\ell+1)}\bigr)\bigr]\notag\\
    &\leq C(\Delta+\Delta E_n^\delta)\,\bigl[1 + \EB[(1+|X_n|^2)^{p/2}] + \EB[|X_{n-J}|^{(\ell+1)p}]\bigr].\label{eq:onestep-leading}
\end{align}
By \eqref{eq:driftstep}, the deterministic second-order term satisfies
\begin{equation}\label{eq:onestep-second}
    \tfrac{1}{2}\EB\bigl[|(\Delta h_{\Delta,n}+\Delta E_n^\delta f_n)^\top \nabla^2 V(X_n)(\Delta h_{\Delta,n}+\Delta E_n^\delta f_n)|\bigr] \leq C(\Delta+\Delta E_n^\delta)\EB[(1+|X_n|^2)^{p/2}].
\end{equation}
By \eqref{eq:driftstep} and \eqref{eq:diffusionstep}, since $|\nabla^3V(X_n)| \leq C(1+|X_n|^{(p-3)\vee0}) \leq C(1+|X_n|^2)^{p/2}$, the non-vanishing third-order terms are bounded by
\begin{align}
    &\quad C\,\EB\bigl[(1+|X_n|^{(p-3)\vee0})\bigl(|\Delta h_{\Delta,n}+\Delta E_n^\delta f_n|^3 + |\Delta h_{\Delta,n}+\Delta E_n^\delta f_n|\Delta E_n^\delta|g_n|^2\bigr)\bigr]\notag\\
    &\leq C(\Delta+\Delta E_n^\delta)\EB[(1+|X_n|^2)^{p/2}].\label{eq:onestep-third}
\end{align}
For the remainder $R_4$, when $p\geq4$, applying \eqref{eq:driftstep}, \eqref{eq:diffusionstep}, and the Gaussian moment bound gives
\begin{align}
    \EB[|R_4|]
    &\leq C\,\EB\bigl[(1+|X_n|^{p-4})\bigl(|\Delta h_{\Delta,n}+\Delta E_n^\delta f_n|^4+(\Delta E_n^\delta|g_n|^2)^2\bigr)\notag\\
    &\hspace{13mm} + |\Delta h_{\Delta,n}+\Delta E_n^\delta f_n|^p + (\Delta E_n^\delta|g_n|^2)^{p/2}\bigr]\notag\\
    &\leq C(\Delta+\Delta E_n^\delta)\EB[(1+|X_n|^2)^{p/2}].\label{eq:onestep-remainder}
\end{align}
When $2\leq p<4$, the fourth-order derivatives $\nabla^4 V$ are globally bounded, which gives $\EB[|R_4|] \leq C(\Delta+\Delta E_n^\delta)$.
Taking $\EB$ in the Taylor expansion of $V(X_{n+1})$ and substituting \eqref{eq:onestep-leading}--\eqref{eq:onestep-remainder} proves the assertion for $\Delta+\Delta E_n^\delta\leq1$.

Next, we consider the case $\Delta+\Delta E_n^\delta > 1$.
Applying the inequality $(1+|u+v+w|^2)^{p/2} \leq C(1+|u|^2)^{p/2} + C|v|^p + C|w|^p$ directly to the ETEM update, together with $|\Delta h_{\Delta,n}+\Delta E_n^\delta f_n|\leq\sqrt{2(\Delta+\Delta E_n^\delta)}$ from \eqref{eq:driftstep} and the Gaussian moment bound $\EB[|g_n\Delta B_n|^p]\leq C\EB[(\Delta E_n^\delta|g_n|^2)^{p/2}]\leq C(\Delta+\Delta E_n^\delta)^{p/4}\leq C(\Delta+\Delta E_n^\delta)^{p/2}$, yields
\begin{align}
    \EB[(1+|X_{n+1}|^2)^{p/2}]
    &\leq C\bigl[\EB[(1+|X_n|^2)^{p/2}] + (\Delta+\Delta E_n^\delta)^{p/2}\bigr]\notag\\
    &\leq \EB[(1+|X_n|^2)^{p/2}] + C\bigl((\Delta+\Delta E_n^\delta)+(\Delta+\Delta E_n^\delta)^{p/2}\bigr)\notag\\
    &\quad\times \bigl[1 + \EB[(1+|X_n|^2)^{p/2}] + \EB[|X_{n-J}|^{(\ell+1)p}]\bigr].\label{eq:onestep-large}
\end{align}
Combining the estimates for $\Delta+\Delta E_n^\delta\leq1$ with \eqref{eq:onestep-large} for $\Delta+\Delta E_n^\delta>1$ completes the proof for every $p\geq2$.
\end{proof}

\begin{proof}[Proof of \Cref{thm:moments}]
It is straightforward to verify that for every fixed $p\geq2$ and all $a>0$,
\begin{equation}\label{eq:log}
    \log(1+C(a+a^{p/2}))\leq Ca.
\end{equation}
Thus, we have
\begin{align}\label{eq:cell-product}
    \prod_{n=0}^{N-1}\left[1+C\bigl((\Delta+\Delta E_n^\delta) +(\Delta+\Delta E_n^\delta)^{p/2}\bigr)\right]&\leq e^{C(T+E_N^\delta)}\notag\\
    &\leq e^{CA(T)}.
\end{align}

Fix the interval $[(k-1)\tau,k\tau]$, corresponding to $(k-1)J\leq n<kJ$.
Suppose the moments of the numerical solution of every finite order have been bounded up to $(k-1)\tau$.
Since $n-J<(k-1)J$, the delayed states belong to $[-\tau,(k-1)\tau]$.
Iterating \Cref{lem:onestep} across $(k-1)J\leq i\leq n$ yields, for $(k-1)J\leq n<kJ$,
\begin{align*}
    &\quad \EB[|X_{n+1}|^p]\\
    &\leq 1+\EB[(1+|X_{n+1}|^2)^{p/2}]+\max_{-J\leq j\leq(k-1)J}\EB[|X_j|^{(\ell+1)p}]\\
    &\leq [1+C((\Delta+\Delta E_n^\delta)+(\Delta+\Delta E_n^\delta)^{p/2})]\left(1+\EB[(1+|X_n|^2)^{p/2}]+\max_{-J\leq j\leq(k-1)J}\EB[|X_j|^{(\ell+1)p}]\right)\\
    &\leq \prod_{i=(k-1)J}^n\left[1+C\bigl((\Delta+\Delta E_i^\delta)+(\Delta+\Delta E_i^\delta)^{p/2}\bigr)\right]\\
    &\quad\times\left(1+\EB[(1+|X_{(k-1)J}|^2)^{p/2}]+\max_{-J\leq j\leq(k-1)J}\EB[|X_j|^{(\ell+1)p}]\right)\\
    &\leq Ce^{CA(T)},
\end{align*}
where the last inequality follows from \eqref{eq:cell-product} and the moment bounds established on $[-\tau,(k-1)\tau]$.
Taking the maximum over $(k-1)J\leq n\leq kJ$ gives $\max_{(k-1)J\leq n\leq kJ}\EB[|X_n|^p]\leq Ce^{CA(T)}$.
Since the initial history on $[-\tau,0]$ is bounded by \Cref{ass:history}, repeating this argument across successive delay intervals up to $T$ yields
\begin{equation*}
    \max_{-J\leq n\leq N}\EB[|X_n|^p]\leq Ce^{CA(T)}.
\end{equation*}
By \eqref{eq:growth} and the fact that taming decreases coefficient norms, this also bounds the moments of $h_{\Delta,n},f_n,g_n$ and the original coefficients uniformly in $n$.

On $(E_n^\delta,E_{n+1}^\delta]$, define the Brownian integrand indexed by $u$ to be $g_n$, with empty intervals omitted.
This integrand is predictable, since $g_n$ is known at $E_n^\delta$.
Applying H\"older's inequality to the drift sums and the BDG inequality to the Brownian integral in \eqref{eq:fd-scheme}, we obtain
\begin{align*}
    &\quad\EB\left[\max_{0\leq j\leq N}|X_j|^p\right]\\*
    &\leq C+CA(T)^{p-1}\sum_{n=0}^{N-1}\bigl(\Delta\EB[|h_{\Delta,n}|^p]+\Delta E_n^\delta\EB[|f_n|^p]\bigr)\\*
    &\quad+C(E_N^\delta)^{p/2-1}\sum_{n=0}^{N-1}\Delta E_n^\delta\EB[|g_n|^p]\\*
    &\leq C[1+A(T)^p+(E_N^\delta)^{p/2}]e^{CA(T)}\\*
    &\leq Ce^{CA(T)}.
\end{align*}
The diffusion term vanishes when $E_N^\delta=0$.
The same calculation between nodes $i$ and $j$, for $0\leq i\leq j$, gives
\begin{equation*}
    \EB[|X_j-X_i|^p]\leq Ce^{CA(T)}\bigl[(t_j-t_i)^p+(E_j^\delta-E_i^\delta)^p+(E_j^\delta-E_i^\delta)^{p/2}\bigr].
\end{equation*}
The bound $E_j^\delta-E_i^\delta\leq E(T)$ absorbs the linear inverse subordinator increment.
For negative indices, use \Cref{ass:history} and split at zero if necessary.
Taking expectation with respect to $D$ and applying \Cref{lem:inverse} gives the unconditional bound.
\end{proof}

\subsection{Proof of strong convergence}\label{sec:proof-convergence}

The ETEM scheme uses the Brownian increments $B(E_{n+1}^\delta)-B(E_n^\delta)$, whereas the exact solution is driven by $B(E(t))$.
Following \cite{Zuo2026, ZuoGuo2026}, we construct an auxiliary continuous-time process using the coefficients of the ETEM scheme.
This process shares the same drivers as the exact solution, allowing their difference to be estimated using coupled dissipativity.
Its discrepancy from the numerical solution is then controlled by comparing the partitions determined by $E_n$ and $E_n^\delta$.

We define a continuous auxiliary process by setting $\widehat X(t)=\xi(t)$ for $-\tau\leq t\leq0$ and
\begin{equation}\label{eq:fd-lift}
	\widehat X(t)=\xi(0)+\int_0^t h_{\Delta,n(s)}\dd s+\int_0^t f_{n(s)}\dd E(s)+\int_0^t g_{n(s)}\dd B(E(s)),\qquad 0\leq t\leq T.
\end{equation}
Under $\PP_B$, $g_n$ is known at $E_n^\delta\leq E_n$, so the stochastic integrand is predictable.

In the subordinator-grid schemes of \cite{JinKobayashi2019,JinKobayashi2021}, the nodes $\tau_k=D(k\delta)$ satisfy $E(\tau_k)=E^\delta(\tau_k)=k\delta$.
This alignment allows the numerical updates to be represented using the exact clock without comparing two distinct clock partitions.
For the equidistant grid used here, however, $E_n^\delta$ generally differs from $E_n$.
Consequently, the auxiliary reconstruction and the numerical solution may use different sampled coefficients at the same value of the Brownian integration variable $u$.
The following partition comparison controls this mismatch without introducing a factor proportional to the number of time steps.

For the construction of the partitions only, set
\begin{equation*}
    \widehat E_n^\delta=E_n^\delta\quad(0\leq n<N),\qquad \widehat E_N^\delta=E_N.
\end{equation*}
Thus only the last approximate inverse-subordinator cell is extended.
The actual increments $\Delta E_n^\delta$, the coefficients in \eqref{eq:fd-coefficients}, and the recursion \eqref{eq:fd-scheme} are unchanged.
For $0<u\leq E_N$, let $i(u),j(u)\in\{0,\ldots,N-1\}$ satisfy
\begin{equation}\label{eq:partition-indices}
    E_{i(u)}<u\leq E_{i(u)+1},\qquad \widehat E_{j(u)}^\delta<u\leq\widehat E_{j(u)+1}^\delta,
\end{equation}
with zero-length cells omitted.
Values at endpoints have no effect on the integrals below.

\begin{lemma}\label{lem:ordered-partitions}
Let $(E_n)$ and $(E_n^\delta)$ be nondecreasing sequences satisfying $E_0=E_0^\delta=0$ and $0\leq E_k-E_k^\delta\leq\delta$ for $1\leq k\leq N$, with the extension and indices defined above.
For every nondecreasing real sequence $(c_n)_{n=0}^{N-1}$ and $z>0$,
\begin{equation*}
    \int_0^{E_N}|c_{j(u)}-c_{i(u)}|^z\dd u\leq|c_{N-1}-c_0|^z\begin{cases}E_N^{1-z}\delta^z,&0<z\leq1,\\ \delta,&z>1.\end{cases}
\end{equation*}
The set $U=\{u\in(0,E_N]: i(u)\ne j(u)\}$ represents the mismatch region where the exact and approximate partitions assign different time-step indices. 
For $r\in\{1/2,1\}$, we have
\begin{equation*}
    \int_U\bigl[(\Delta E_{i(u)}^\delta)^r+(\Delta E_{j(u)}^\delta)^r\bigr]\dd u\leq C\delta^r E_N,
\end{equation*}
where $C$ is independent of the endpoint sequences, $N$, and $\delta$.
\end{lemma}

\begin{proof}
If $E_N=0$, both integrals vanish.
Assume $E_N>0$.
Since $E_k^\delta\leq E_k$ at every interior endpoint, $i(u)\leq j(u)$.
For $1\leq k<N$, the relation $i(u)<k\leq j(u)$ is equivalent to $E_k^\delta<u\leq E_k$, up to endpoints.
Consequently, we have
\begin{equation*}
    c_{j(u)}-c_{i(u)}=\sum_{k=1}^{N-1}(c_k-c_{k-1})\mathbf1_{(E_k^\delta,E_k]}(u)
\end{equation*}
for almost every $u$.
By monotonicity of $(c_n)$,
\begin{align*}
    \int_0^{E_N}|c_{j(u)}-c_{i(u)}|\dd u&=\sum_{k=1}^{N-1}(c_k-c_{k-1})(E_k-E_k^\delta)\\
    &\leq\delta(c_{N-1}-c_0).
\end{align*}
If $c_{N-1}=c_0$, the first estimate is immediate.
Otherwise, Jensen's inequality yields the estimate for $0<z\leq1$, while $|c_{j(u)}-c_{i(u)}|^z\leq|c_{N-1}-c_0|^{z-1}|c_{j(u)}-c_{i(u)}|$ gives the bound for $z>1$.

To establish the second estimate, first note that for $r\in\{1/2,1\}$ and $x,y\geq0$,
\begin{equation*}
    (x+\delta)^r\min(x,\delta)\leq2^r\delta^rx,\qquad y^r\min(y,\delta)\leq\delta^ry.
\end{equation*}
On an exact cell $(E_i,E_{i+1}]$, a mismatch occurs only if $u>E_{i+1}^\delta$.
Hence its intersection with $U$ has length at most $\min(E_{i+1}-E_i,\delta)$, and $\Delta E_i^\delta\leq E_{i+1}-E_i+\delta$.
On an extended approximate cell $(\widehat E_j^\delta,\widehat E_{j+1}^\delta]$, a mismatch occurs only if $u\leq E_j$.
Its intersection with $U$ therefore has length at most $\min(\widehat E_{j+1}^\delta-\widehat E_j^\delta,\delta)$, and $\Delta E_j^\delta\leq \widehat E_{j+1}^\delta-\widehat E_j^\delta$.
Splitting the integral over the two partitions and applying these bounds, we obtain
\begin{align*}
    &\quad\int_U\bigl[(\Delta E_{i(u)}^\delta)^r+(\Delta E_{j(u)}^\delta)^r\bigr]\dd u\\
    &=\sum_{i=0}^{N-1}(\Delta E_i^\delta)^r\int_{U\cap(E_i,E_{i+1}]}1\dd u+\sum_{j=0}^{N-1}(\Delta E_j^\delta)^r\int_{U\cap(\widehat E_j^\delta,\widehat E_{j+1}^\delta]}1\dd u\\
    &\leq\sum_{i=0}^{N-1}(E_{i+1}-E_i+\delta)^r\min(E_{i+1}-E_i,\delta)+\sum_{j=0}^{N-1}(\widehat E_{j+1}^\delta-\widehat E_j^\delta)^r\min(\widehat E_{j+1}^\delta-\widehat E_j^\delta,\delta)\\
    &\leq2^r\delta^r\sum_{i=0}^{N-1}(E_{i+1}-E_i)+\delta^r\sum_{j=0}^{N-1}(\widehat E_{j+1}^\delta-\widehat E_j^\delta)\\
    &\leq C\delta^rE_N.\qedhere
\end{align*}
\end{proof}

The preceding comparison of the two partitions and the increment bounds for $X_n$ give the following coefficient estimates.
\begin{lemma}\label{lem:partition-coefficients}
Under \Cref{ass:polynomial,ass:current,ass:time,ass:history}, for every $p\geq2$ and $\kappa=\min\{\gamma,\beta_1,\beta_2,1/2\}$, we have
\begin{align*}
    \int_0^{E_N}\left(\EB[|f_{j(u)}-f_{i(u)}|^{p}]\right)^{1/p}\dd u &\leq Ce^{CA(T)}\delta^\kappa,\\
    \int_0^{E_N}\left(\EB[|g_{j(u)}-g_{i(u)}|^{p}]\right)^{2/p}\dd u &\leq Ce^{CA(T)}\delta^{2\kappa}.
\end{align*}
\end{lemma}

\begin{proof}
Let $\varphi\in\{f,g\}$.
For fixed $u$, write $i=i(u)$ and $j=j(u)$, with negative inverse subordinator indices set to zero.
By \Cref{ass:polynomial,ass:time}, H\"older's inequality, and \Cref{thm:moments}, we have
\begin{align*}
    &\quad \left(\EB[|\varphi(X_j)-\varphi(X_i)|^{p}]\right)^{1/p} \\
    &\leq Ce^{CA(T)}\bigl[(t_j-t_i)^{\beta_1}+(E_j^\delta-E_i^\delta)^{\beta_2}+\left(\EB[|X_j-X_i|^{2p}]\right)^{1/(2p)}+\left(\EB[|X_{j-J}-X_{i-J}|^{2p}]\right)^{1/(2p)}\bigr]\\
    &\leq Ce^{CA(T)}\bigl[(t_j-t_i)^{\beta_1}+(E_j^\delta-E_i^\delta)^{\beta_2}+(t_j-t_i)+(t_j-t_i)^\gamma+(E_j^\delta-E_i^\delta)^{1/2}+(E_{j-J}^\delta-E_{i-J}^\delta)^{1/2}\bigr].
\end{align*}
Since $0\leq t_j-t_i\leq T$, it follows that
\begin{align*}
    &\quad\left(\EB[|\varphi(X_j)-\varphi(X_i)|^{p}]\right)^{2/p}\\
    &\leq Ce^{CA(T)}\bigl[(t_j-t_i)^{2\beta_1}+(t_j-t_i)^{2\gamma}+(E_j^\delta-E_i^\delta)^{2\beta_2}+(E_j^\delta-E_i^\delta)+(E_{j-J}^\delta-E_{i-J}^\delta)\bigr].
\end{align*}
We integrate over $[0,E_N]$ and apply \Cref{lem:ordered-partitions} to the monotone sequences $t_n$, $E_n^\delta$, and $E_{n-J}^\delta$.
Their endpoint differences are bounded by $T$, $E_N$, and $E_N$, respectively.
Absorbing polynomial factors in $A(T)$ into $e^{CA(T)}$, we obtain
\begin{equation}\label{eq:untamed-partitions}
    \int_0^{E_N}\left(\EB[|\varphi(X_{j(u)})-\varphi(X_{i(u)})|^{p}]\right)^{2/p}\dd u \leq Ce^{CA(T)}\delta^{2\kappa}.
\end{equation}

By the taming bound \eqref{eq:tamebias} and \Cref{thm:moments}, we have $(\EB[|f_n-f(X_n)|^p])^{1/p}\leq Ce^{CA(T)}\sqrt{\Delta E_n^\delta}$.
Since $i(u)=j(u)$ outside $U$, we have $f_j-f_i = f(X_j)-f(X_i)$ on $[0,E_N]\setminus U$.
Applying Minkowski's inequality, Cauchy--Schwarz, \eqref{eq:untamed-partitions}, and \Cref{lem:ordered-partitions} with $r=1/2$, we obtain
\begin{align*}
	&\quad\int_0^{E_N}(\EB[|f_j-f_i|^p])^{1/p}\dd u\\
	&\leq \int_0^{E_N}(\EB[|f(X_j)-f(X_i)|^p])^{1/p}\dd u + \int_U\bigl((\EB[|f_j-f(X_j)|^p])^{1/p}+(\EB[|f_i-f(X_i)|^p])^{1/p}\bigr)\dd u\\*
	&\leq E_N^{1/2} \left(\int_0^{E_N}(\EB[|f(X_j)-f(X_i)|^p])^{2/p}\dd u\right)^{1/2}+Ce^{CA(T)}\int_U(\sqrt{\Delta E_i^\delta}+\sqrt{\Delta E_j^\delta})\dd u\\*
	&\leq Ce^{CA(T)}\delta^\kappa.
\end{align*}
Similarly, decomposing $g_j-g_i$ and applying \eqref{eq:tamebias}, \eqref{eq:untamed-partitions}, and \Cref{lem:ordered-partitions} with $r=1$, one readily obtains the diffusion bound.
\end{proof}

\begin{proposition}\label{prop:lift-defect}
Under \Cref{ass:polynomial,ass:current,ass:time,ass:history}, for every $p\geq2$, we obtain
\begin{align*}
    \EB\left[\sup_{-\tau\leq t\leq T}|\widehat X(t)|^p\right] &\leq Ce^{CA(T)},\\
    \EB\left[\max_{0\leq n\leq N}|\widehat X(t_n)-X_n|^p\right] &\leq Ce^{CA(T)}\varepsilon_\delta^p.
\end{align*}
\end{proposition}

\begin{proof}
For the first assertion, apply H\"older's inequality and the BDG inequality to \eqref{eq:fd-lift}.
By \eqref{eq:growth} and \Cref{thm:moments} at order $(\ell+1)p$,
\begin{align*}
    &\quad\EB\left[\sup_{0\leq t\leq T}|\widehat X(t)|^p\right]\\
    &\leq C+CA(T)^{p-1}\int_0^T\EB[|h_{\Delta,n(s)}|^p+|f_{n(s)}|^p]\dd A(s)+CE(T)^{p/2-1}\int_0^T\EB[|g_{n(s)}|^p]\dd E(s)\\
    &\leq C[1+A(T)^p+E(T)^{p/2}]e^{CA(T)}\\
    &\leq Ce^{CA(T)}.
\end{align*}
Since $\widehat X=\xi$ on $[-\tau,0]$, the first estimate follows.
For the second estimate, define the following martingales
\begin{equation*}
    M_a(u)=\int_0^u g_{i(v)}\dd B(v),\qquad M_b(u)=\int_0^u g_{j(v)}\dd B(v),\qquad 0\leq u\leq E_N.
\end{equation*}

Subtracting \eqref{eq:fd-scheme}, summed over $k=0,\ldots,n-1$, from \eqref{eq:fd-lift} at $t_n$ and cancelling the $\dd t$ drift terms, we obtain
\begin{equation}
    \widehat X(t_n)-X_n =\int_0^{E_n}(f_{i(u)}-f_{j(u)})\dd u +\int_{E_n^\delta}^{E_n}f_{j(u)}\dd u+(M_a-M_b)(E_n)+M_b(E_n)-M_b(E_n^\delta).\label{eq:lift-identity}
\end{equation}
Taking the maximum over $0\leq n\leq N$, raising to the $p$th power, and applying $\EB$, we obtain
\begin{align}
	&\quad \EB\left[\max_{0\leq n\leq N}|\widehat X(t_n)-X_n|^p\right]\notag\\
	&\leq C\EB\left[\left(\int_0^{E_N}|f_{i(u)}-f_{j(u)}|\dd u\right)^p\right]+C\EB\left[\max_{0\leq n\leq N}\left|\int_{E_n^\delta}^{E_n}f_{j(u)}\dd u\right|^p\right]\notag\\
	&\quad+C\EB\left[\sup_{0\leq u\leq E_N}|(M_a-M_b)(u)|^p\right]+C\EB\left[\max_{0\leq n\leq N}|M_b(E_n)-M_b(E_n^\delta)|^p\right].\label{eq:lift-error-decomposition}
\end{align}
By Minkowski's inequality and \Cref{lem:partition-coefficients},
\begin{align}
    \EB\left[\left(\int_0^{E_N}|f_{i(u)}-f_{j(u)}|\dd u\right)^p\right]&\leq\left(\int_0^{E_N}(\EB[|f_{i(u)}-f_{j(u)}|^p])^{1/p}\dd u\right)^p\notag\\
    &\leq Ce^{CA(T)}\delta^{\kappa p}.\label{eq:lift-drift-mismatch}
\end{align}
By \eqref{eq:growth} and \Cref{thm:moments}, the second satisfies
\begin{align}
    \EB\left[\max_n\left|\int_{E_n^\delta}^{E_n}f_{j(u)}\dd u\right|^p\right]&\leq\delta^p\EB\left[\max_{0\leq k<N}|f_k|^p\right]\notag\\
    &\leq Ce^{CA(T)}\delta^p.\label{eq:lift-drift-endpoint}
\end{align}
By the BDG and Minkowski inequalities and \Cref{lem:partition-coefficients},
\begin{align}
    \EB\left[\sup_{0\leq u\leq E_N}|(M_a-M_b)(u)|^p\right]&\leq C\left(\int_0^{E_N}(\EB[|g_{i(u)}-g_{j(u)}|^p])^{2/p}\dd u\right)^{p/2}\notag\\*
    &\leq Ce^{CA(T)}\delta^{\kappa p}.\label{eq:lift-martingale}
\end{align}

To control the endpoint increments of $M_b$, extend its integrand by zero to $[0,L]$, where $L=1+E_N$.
By $|g_n|\leq|g(X_n)|$, \eqref{eq:growth}, and \Cref{thm:moments}, the $2p$th moment of its supremum is bounded by $Ce^{CA(T)}$.
Since $0\leq E_n-E_n^\delta\leq\delta$, applying \Cref{lem:martingale-modulus} gives
\begin{align}
	&\quad \EB\left[\max_{0\leq n\leq N}|M_b(E_n)-M_b(E_n^\delta)|^p\right]\notag\\
	&\leq \EB\left[\sup_{\substack{0\leq u<v\leq L\\v-u\leq\delta}}|M_b(v)-M_b(u)|^p\right]\notag\\
	&\leq Ce^{CA(T)}[\delta\log(e+L/\delta)]^{p/2}\notag\\
	&\leq Ce^{CA(T)}[\delta(1+|\log\delta|)]^{p/2}.\label{eq:lift-endpoint}
\end{align}

Substituting \eqref{eq:lift-drift-mismatch}--\eqref{eq:lift-endpoint} into \eqref{eq:lift-error-decomposition} and using $\delta^\kappa\leq\varepsilon_\delta$ and $\sqrt{\delta(1+|\log\delta|)}\leq\varepsilon_\delta$, we obtain the second estimate.
\end{proof}

With the moment bounds and partition-mismatch estimates established, we now analyze the local increments and residuals, proving the conditional error recursion and the strong convergence theorem.

\begin{lemma}\label{lem:local}
Under \Cref{ass:polynomial,ass:current,ass:time,ass:history}, for every $p\geq2$ and $s\in[0,T]$,
\begin{align*}
    \EB[|\widehat X(s)-\widetilde X(s)|^p]&\leq Ce^{CA(T)}[(\Delta+\Xi_\Delta)^{p/2}+\varepsilon_\delta^p],\\
    \EB[|\widehat X(s-\tau)-\widetilde X(s-\tau)|^p]&\leq Ce^{CA(T)}[\Delta^{\gamma p}+(\Delta+\Xi_\Delta)^{p/2}+\varepsilon_\delta^p].
\end{align*}
\end{lemma}
\begin{proof}
For $s\in[t_n,t_{n+1})$, subtracting \eqref{eq:fd-lift} at $t_n$ gives
\begin{equation*}
    \widehat X(s)-\widehat X(t_n)=(s-t_n)h_{\Delta,n}+(E(s)-E_n)f_n+g_n[B(E(s))-B(E_n)].
\end{equation*}
By \eqref{eq:growth}, and \Cref{thm:moments},
\begin{align*}
    \EB[|\widehat X(s)-\widehat X(t_n)|^p]&\leq Ce^{CA(T)}[\Delta^p+(E(s)-E_n)^p+(E(s)-E_n)^{p/2}]\\
    &\leq Ce^{CA(T)}[(\Delta+\Xi_\Delta)^p+\Xi_\Delta^{p/2}]\\
    &\leq Ce^{CA(T)}(\Delta+\Xi_\Delta)^{p/2}.
\end{align*}
The last inequality uses $\Delta+\Xi_\Delta\leq A(T)$ and absorbs the remaining polynomial factor into the exponential.
Since $\widetilde X(s)=X_n$ for $s\in[t_n,t_{n+1})$ (with $\widetilde X(T)=X_N$), the decomposition
\begin{equation*}
	\widehat X(s)-\widetilde X(s)=[\widehat X(s)-\widehat X(t_n)]+[\widehat X(t_n)-X_n]
\end{equation*}
and \Cref{prop:lift-defect} establish the first bound for all $s\in[0,T]$.

When $s\in[\tau,T]$, since $s-\tau\in[0,T]$, the estimate follows directly from the first assertion.
When $s\in[0,\tau)$, we have $\widehat X(s-\tau)=\xi(s-\tau)$ and $\widetilde X(s-\tau)=\xi(\fl{s}-\tau)$.
By \Cref{ass:history},
\begin{equation*}
    |\widehat X(s-\tau)-\widetilde X(s-\tau)|=|\xi(s-\tau)-\xi(\fl{s}-\tau)|\leq L_2\Delta^\gamma.
\end{equation*}
Combining these cases proves the delayed estimate.
\end{proof}

To combine time-step discretization, coefficient regularity, and taming in a single error equation, define
\begin{align*}
    \mathcal R_h(s)&=h(s,E(s),\widehat X(s),\widehat X(s-\tau))-h_{\Delta,n(s)},\\
    \mathcal R_f(s)&=f(s,E(s),\widehat X(s),\widehat X(s-\tau))-f_{n(s)},\\
    \mathcal R_g(s)&=g(s,E(s),\widehat X(s),\widehat X(s-\tau))-g_{n(s)}.
\end{align*}
The time-grid contribution is measured by
\begin{equation}\label{eq:errorscale}
    \varepsilon_\Delta=\Delta^\gamma+\Delta^{\beta_1}+\Xi_\Delta^{\beta_2}+(\Delta+\Xi_\Delta)^{1/2}.
\end{equation}
This quantity depends on $D$ but is deterministic under $\PP_B$.
Together with \eqref{eq:clock-error-scale}, it controls the residuals in the following lemma.

\begin{lemma}\label{lem:residual}
Under \Cref{ass:polynomial,ass:current,ass:time,ass:history}, for every $p\geq2$,
\begin{equation*}
    \int_0^T\EB[|\mathcal R_h(s)|^p]\dd s+\int_0^T\EB[|\mathcal R_f(s)|^p+|\mathcal R_g(s)|^p]\dd E(s)\leq Ce^{CA(T)}(\varepsilon_\Delta+\varepsilon_\delta)^p.
\end{equation*}
\end{lemma}
\begin{proof}
For $s\in[t_n,t_{n+1})$, \eqref{eq:clock-bracket} gives
\begin{equation*}
    0\leq E(s)-E_n^\delta\leq\Xi_\Delta+\delta,\qquad \Delta E_n^\delta\leq\Xi_\Delta+\delta.
\end{equation*}
By \Cref{ass:polynomial,ass:time}, Cauchy--Schwarz, \eqref{eq:tamebias}, \Cref{thm:moments}, and \Cref{prop:lift-defect},
\begin{align*}
    &\quad\sum_{\varphi\in\{h,f,g\}}\EB[|\mathcal R_\varphi(s)|^p]\\
    &\leq Ce^{CA(T)}\bigl[(\EB[|\widehat X(s)-\widetilde X(s)|^{2p}])^{1/2}+(\EB[|\widehat X(s-\tau)-\widetilde X(s-\tau)|^{2p}])^{1/2}\\
    &\hspace{18mm}+\Delta^{\beta_1p}+(\Xi_\Delta+\delta)^{\beta_2p}+\Delta^{p/2}+(\Xi_\Delta+\delta)^{p/2}\bigr]\\
    &\leq Ce^{CA(T)}\bigl[\Delta^{\gamma p}+\Delta^{\beta_1p}+\Xi_\Delta^{\beta_2p}+(\Delta+\Xi_\Delta)^{p/2}+\varepsilon_\delta^p+\delta^{\beta_2p}+\delta^{p/2}\bigr]\\
    &\leq Ce^{CA(T)}(\varepsilon_\Delta+\varepsilon_\delta)^p.
\end{align*}
Here, the second inequality applies \Cref{lem:local} at order $2p$, while the last step follows from $\delta^{\beta_2}+\sqrt\delta\leq\varepsilon_\delta$.
\end{proof}

\begin{proof}[Proof of \Cref{thm:fully-discrete}]
Let $e(t)=X(t)-\widehat X(t)$ for $t\in[-\tau,T]$.
The auxiliary process has the exact initial history, so $e(t)=0$ for $-\tau\leq t\leq0$.
For each $\varphi\in\{h,f,g\}$, separate the current-state difference and the delayed-state and residual terms as
\begin{align*}
    d_\varphi(s)&=\varphi(s,E(s),X(s),X(s-\tau))-\varphi(s,E(s),\widehat X(s),X(s-\tau)),\\
    r_\varphi(s)&=\varphi(s,E(s),\widehat X(s),X(s-\tau))-\varphi(s,E(s),\widehat X(s),\widehat X(s-\tau))+\mathcal R_\varphi(s).
\end{align*}
The error process $e(s)$ then satisfies
\begin{equation*}
    \dd e(s)=(d_h(s)+r_h(s))\dd s+(d_f(s)+r_f(s))\dd E(s)+(d_g(s)+r_g(s))\dd B(E(s)).
\end{equation*}
By \Cref{ass:current}, the $\mathrm{d}t$ drift satisfies $\ip{e(s)}{d_h(s)}\leq K_1|e(s)|^2$.
Using \Cref{ass:current} with parameter $2p+1$, we have
\begin{equation*}
    \ip{e(s)}{d_f(s)}+p|d_g(s)|^2\leq C|e(s)|^2.
\end{equation*}
Meanwhile, applying \Cref{ass:polynomial} to the delayed difference in $r_\varphi(s)$ yields
\begin{equation*}
    |r_\varphi(s)| \leq C(1+|\widehat X(s)|^\ell+|X(s-\tau)|^\ell+|\widehat X(s-\tau)|^\ell)|e(s-\tau)| + |\mathcal R_\varphi(s)|.
\end{equation*}
By \eqref{eq:growth} and \Cref{prop:exact,prop:lift-defect}, the local martingale in It\^o's formula for $|e(t)|^p$ is a true square-integrable martingale under $\PP_B$.
Applying It\^o's formula to $|e(t)|^p$, we bound the diffusion correction by $p(p-1)|e(s)|^{p-2}(|d_g(s)|^2+|r_g(s)|^2)$.
The coupled bound above leaves the negative contribution $-p|e(s)|^{p-2}|d_g(s)|^2$.
Young's inequality for the remaining terms therefore gives
\begin{equation}\label{eq:error-decomp}
	|e(t)|^p+pJ(t)\leq F(t)+M(t),
\end{equation}
where
\begin{align*}
	J(t)&=\int_0^t|e(s)|^{p-2}|d_g(s)|^2\dd E(s),\\
	F(t)&=C\int_0^t|e(s)|^p\dd A(s)+C\int_0^t|r_h(s)|^p\dd s+C\int_0^t(|r_f(s)|^p+|r_g(s)|^p)\dd E(s),\\
	M(t)&=p\int_0^t|e(s)|^{p-2}\ip{e(s)}{(d_g(s)+r_g(s))\dd B(E(s))}.
\end{align*}
Here, the factor $|e(s)|^{p-2}$ is understood as $1$ when $p=2$.
Taking $\EB$ on both sides yields $p\EB[J(t)]\leq\EB[F(t)]$ since $\EB[M(t)]=0$.
By the BDG and Young inequalities,
\begin{align}\label{eq:error-m-bound}
	\EB\left[\sup_{0\leq u\leq t}|M(u)|\right]&\leq\frac12\EB\left[\sup_{0\leq u\leq t}|e(u)|^p\right]+C\int_0^t\EB[|e(s)|^{p-2}|d_g(s)+r_g(s)|^2]\dd E(s)\notag\\
	&\leq\frac12\EB\left[\sup_{0\leq u\leq t}|e(u)|^p\right]+C\EB[J(t)]+C\int_0^t\EB[|e(s)|^p+|r_g(s)|^p]\dd E(s)\notag\\
	&\leq\frac12\EB\left[\sup_{0\leq u\leq t}|e(u)|^p\right]+C\EB[F(t)].
\end{align}

Taking the supremum over $u\in[0,t]$ in \eqref{eq:error-decomp}, dropping $pJ(t)\geq0$, and taking expectation $\EB$, we use \eqref{eq:error-m-bound} to absorb the martingale contribution, which yields
\begin{equation}\label{eq:error-sup-f}
    \EB\left[\sup_{0\leq u\leq t}|e(u)|^p\right]\leq C\EB[F(t)].
\end{equation}
For $s\leq k\tau$ with $1\leq k\leq K$, the delayed argument satisfies $s-\tau\in[-\tau,(k-1)\tau]$.
Applying the Cauchy--Schwarz inequality together with the uniform moment bounds in \Cref{prop:exact,prop:lift-defect} to the delayed term in $|r_\varphi(s)|^p$, we obtain
\begin{align}\label{eq:delayed-term-bound}
	&\quad\EB\bigl[ (1+|\widehat X(s)|^\ell+|X(s-\tau)|^\ell+|\widehat X(s-\tau)|^\ell)^p |e(s-\tau)|^p\bigr]\notag\\*
	&\leq \bigl(\EB[(1+|\widehat X(s)|^\ell+|X(s-\tau)|^\ell +|\widehat X(s-\tau)|^\ell)^{2p}]\bigr)^{1/2} \left(\EB\left[\sup_{-\tau\leq u\leq(k-1)\tau}|e(u)|^{2p}\right]\right)^{1/2}\notag\\*
	&\leq Ce^{CA(T)}\left(\EB\left[\sup_{-\tau\leq u\leq(k-1)\tau}|e(u)|^{2p}\right]\right)^{1/2}.
\end{align}
Substituting \eqref{eq:delayed-term-bound} into \eqref{eq:error-sup-f} and applying \Cref{lem:delay,lem:residual}, we obtain for $t\leq k\tau$
\begin{align*}
    \EB\left[\sup_{0\leq u\leq t}|e(u)|^p\right] &\leq C\int_0^t \EB\left[\sup_{0\leq u\leq s}|e(u)|^p\right]\dd A(s)\\*
    &\quad+Ce^{CA(T)}\left[A(T)\left(\EB\left[\sup_{-\tau\leq u\leq(k-1)\tau}|e(u)|^{2p}\right]\right)^{1/2}+(\varepsilon_\Delta+\varepsilon_\delta)^p\right].
\end{align*}
Applying \Cref{lem:stieltjes-gronwall} and absorbing $A(T)$ into the exponential, we obtain
\begin{equation}\label{eq:errorrecursion}
    \EB\left[\sup_{-\tau\leq t\leq k\tau}|e(t)|^p\right] \leq Ce^{CA(T)} \left[(\varepsilon_\Delta+\varepsilon_\delta)^p+\left(\EB\left[\sup_{-\tau\leq u\leq(k-1)\tau}|e(u)|^{2p}\right]\right)^{1/2}\right].
\end{equation}
Since $e(t)=0$ on $[-\tau,0]$, the delayed term vanishes for $k=1$, which gives
\begin{equation*}
	\EB\left[\sup_{-\tau\leq t\leq\tau}|e(t)|^p\right]\leq Ce^{CA(T)}(\varepsilon_\Delta+\varepsilon_\delta)^p.
\end{equation*}
Iterating \eqref{eq:errorrecursion} over $k=1,\dots,K$, simultaneously for all $p\geq2$, we obtain
\begin{equation*}
	\EB\left[\sup_{0\leq t\leq T}|e(t)|^p\right]\leq Ce^{CA(T)}(\varepsilon_\Delta+\varepsilon_\delta)^p,
\end{equation*}
since $T=K\tau$.
Splitting the nodal error and applying \Cref{prop:lift-defect} gives
\begin{align*}
    \EB\left[\max_{0\leq n\leq N}|X(t_n)-X_n|^p\right]&\leq C\EB\left[\sup_{0\leq t\leq T}|e(t)|^p\right]+C\EB\left[\max_{0\leq n\leq N}|\widehat X(t_n)-X_n|^p\right]\\
    &\leq Ce^{CA(T)}(\varepsilon_\Delta+\varepsilon_\delta)^p,
\end{align*}

Taking the total expectation $\EE=\ED[\EB[\cdot]]$ gives
\begin{equation*}
    \EE\left[\max_{0\leq n\leq N}|X(t_n)-X_n|^p\right]\leq C\ED[e^{CA(T)}\varepsilon_\Delta^p]+C\varepsilon_\delta^p\ED[e^{CA(T)}].
\end{equation*}
By \Cref{lem:inverse},
\begin{equation*}
    \ED[e^{CA(T)}\varepsilon_\Delta^p]\leq C\left[\Delta^{\gamma p}+\Delta^{\beta_1p}+\omega_\Delta^{\beta_2p}+\omega_\Delta^{p/2}\right],
\end{equation*}
and $\ED[e^{CA(T)}]<\infty$.
Taking the $p$th root proves the logarithmic bound.
For fixed $0<\rho<\alpha$ and $0<\theta<1/2$, absorbing the logarithmic factors into the corresponding power losses yields the stated power bound.
\end{proof}

\subsection{Proof of the lower bound}\label{sec:proof-lower-bound}

In this subsection, we prove \Cref{thm:lower-bound} by establishing the terminal lower bound for the multiplicative-noise equation with a stable subordinator.

\begin{proof}[Proof of \Cref{thm:lower-bound}]
Since the stable inverse $E(t)$ has a continuous density for each $t>0$, almost surely $E_n^\delta=\delta\lfloor E_n/\delta\rfloor$ for all $0\leq n\leq N$.
The exact solution is given by $X(T)=\mathcal X(E_N)$ with $\mathcal X(u)=\exp\{B(u)-u/2\}$.
For this scalar test model, the ETEM scheme reduces to
\begin{equation*}
	X_{n+1}=X_n+G_n\Delta B_n,\qquad G_n=\frac{X_n}{\sqrt{1+\sqrt{\Delta E_n^\delta}|X_n|^2}},\qquad X_0=1.
\end{equation*}
Setting $G(u)=G_n$ on $(E_n^\delta,E_{n+1}^\delta]$ and $G(u)=0$ on $(E_N^\delta,E_N]$, the terminal error admits the integral representation
\begin{equation*}
	X(T)-X_N=\int_0^{E_N}(\mathcal X(u)-G(u))\dd B(u).
\end{equation*}
For $0\leq s\leq\Delta E_n^\delta$, the martingale property and independent Brownian increments yield
\begin{equation*}
	\EB[(\mathcal X(E_n^\delta+s)-G_n)^2]=e^{E_n^\delta}(e^s-1)+\EB[(\mathcal X(E_n^\delta)-G_n)^2].
\end{equation*}
Integrating over $[0,E_N]$ and applying the It\^o isometry, we obtain
\begin{align}\label{eq:lower-identity}
	\EB[|X(T)-X_N|^2]&=e^{E_N}-e^{E_N^\delta}+\sum_{n=0}^{N-1}e^{E_n^\delta}(e^{\Delta E_n^\delta}-1-\Delta E_n^\delta)+\sum_{n=0}^{N-1}\Delta E_n^\delta\EB[(\mathcal X(E_n^\delta)-G_n)^2]\notag\\
	&\geq E_N-E_N^\delta+\tfrac12\sum_{n=0}^{N-1}(\Delta E_n^\delta)^2,
\end{align}
where the inequality follows from $e^x-1-x\geq x^2/2$, $e^x-e^y\geq x-y$ for $x\geq y\geq 0$, and the nonnegativity of the remaining terms.

By \cite[Lemma~2.1]{GajdaMagdziarz2014} with moment order two, for $0<\Delta\leq T/2$,
\begin{align}
	\EE\left[\sum_{n=0}^{N-1}(E_{n+1}-E_n)^2\right]
	&\geq\frac{2\Delta^{1+\alpha}}{\Gamma(\alpha)\Gamma(2+\alpha)}\sum_{n=1}^{N-1}t_{n+1}^{\alpha-1}\notag\\
	&\geq\frac{T^\alpha}{\Gamma(\alpha)\Gamma(2+\alpha)}\Delta^\alpha.\label{eq:lower-quadratic}
\end{align}
Here, we omit the nonnegative term with $n=0$ and use $t_{n+1}\leq T$ and $N-1\geq T/(2\Delta)$.

By stable scaling, $\PP(E(T)>u)=\PP(D(u)\leq T)=\PP((T/D(1))^\alpha\geq u)$, so that $E(T)\stackrel{d}{=}(T/D(1))^\alpha$.
Since $D(1)$ possesses a smooth density, $E(T)$ also admits a density.
Applying periodic averaging to the fractional part $z\mapsto\{z\}$ with mean $\int_0^1 z\dd z=1/2$, we obtain
\begin{equation}\label{eq:lower-gap}
	\lim_{\delta\to0}\frac{1}{\delta}\EE[E_N-E_N^\delta]=\lim_{\delta\to0}\EE[\{E(T)/\delta\}]=\tfrac12.
\end{equation}
In particular, there exists $\delta_0>0$ such that $\EE[E_N-E_N^\delta]\geq\delta/4$ for all $0<\delta\leq\delta_0$.

To relate the discrete increments $\Delta E_n^\delta$ to the exact increments $E_{n+1}-E_n$, note that $|(E_{n+1}-E_n)-\Delta E_n^\delta|\leq\delta$.
Since $a^2-b^2=(a-b)(a+b)$, summing over $n=0,\dots,N-1$ yields
\begin{align*}
	\left|\sum_{n=0}^{N-1}(E_{n+1}-E_n)^2-\sum_{n=0}^{N-1}(\Delta E_n^\delta)^2\right|&\leq\delta(E_N+E_N^\delta)\leq2\delta E_N.
\end{align*}
Recalling $\EE[E_N]=T^\alpha/\Gamma(1+\alpha)$ and setting $c_2=T^\alpha/[\Gamma(\alpha)\Gamma(2+\alpha)]$ from \eqref{eq:lower-quadratic}, we choose $c_0=c_2/(4\EE[E_N])$.
We now examine two complementary regimes.
If $\delta\leq c_0\Delta^\alpha$, the perturbation is absorbed by $2\delta\EE[E_N]\leq\frac12 c_2\Delta^\alpha$, which implies $\EE[\sum_{n=0}^{N-1}(\Delta E_n^\delta)^2]\geq\frac12 c_2\Delta^\alpha$.
Combining this with \eqref{eq:lower-identity} and \eqref{eq:lower-gap}, we obtain
\begin{equation*}
	\EE[|X(T)-X_N|^2]\geq\tfrac14 c_2\Delta^\alpha+\tfrac14\delta.
\end{equation*}
If instead $\delta>c_0\Delta^\alpha$, the sampling gap bound \eqref{eq:lower-gap} alone suffices, as $\Delta^\alpha<c_0^{-1}\delta$
\begin{equation*}
	\EE[|X(T)-X_N|^2]\geq\frac{\delta}{4}\geq\frac{\Delta^\alpha+\delta}{4(1+c_0^{-1})}.
\end{equation*}
Taking the square root in either case yields $(\EE[|X(T)-X_N|^2])^{1/2}\geq c(\Delta^{\alpha/2}+\delta^{1/2})$, completing the proof.
\end{proof}

\section{Numerical experiments}\label{sec:experiments}

In this section, numerical examples illustrate the strong convergence of the fully discrete ETEM scheme.
Unless otherwise stated, we use the following parameters.
Throughout all simulations, we fix the terminal time $T = 2$ and delay $\tau = 1/2$.
The expectation is approximated by the standard Monte Carlo average over $M = 2000$ independent sample paths.
The test step sizes are chosen as $\Delta \in \{2^{-9}, 2^{-10}, \dots, 2^{-15}\}$, with the subordinator sampling step size set to $\delta = \Delta/8$.
All delay-model approximations are computed by the ETEM scheme \eqref{eq:fd-scheme}.
The reference solution $X_{\rm ref}$ is computed on a fine grid with $\Delta_{\rm ref} = 2^{-19}$ and $\delta_{\rm ref} = 2^{-23}$.
The strong error is evaluated in the grid-maximum root-mean-square sense
\begin{equation*}
    \mathrm{RMSE}(\Delta) = \left( \frac{1}{M} \sum_{i=1}^M \max_{0 \leq n \leq N} \|X_n^{(i)} - X_{\rm ref}^{(i)}(t_n)\|^2 \right)^{1/2}, \qquad N = \frac{T}{\Delta}.
\end{equation*}

\begin{example}\label{ex:polynomial}
We first consider the scalar nonautonomous superlinear TCSDDE
\begin{align*}
    \mathrm{d}X(t) &= h(t, E(t), X(t), X(t-\tau))\,\mathrm{d}t + f(t, E(t), X(t), X(t-\tau))\,\mathrm{d}E(t) \\
    &\quad + g(t, E(t), X(t), X(t-\tau))\,\mathrm{d}B(E(t)),
\end{align*}
where the coefficient functions are given by
\begin{align*}
    h(t, u, x, y) &= \left(0.25 + 0.08\sin(2\pi t) + 0.05\frac{u}{1+u}\right)x - \left(1 + 0.08\cos(\pi t)e^{-u}\right)x^3 + 0.2y^3, \\
    f(t, u, x, y) &= \left(0.20 + 0.06\cos(2\pi t) + 0.04\frac{u}{1+u}\right)x - \left(1 + 0.08\frac{\sin(\pi t)}{1+u}\right)x^5 + 0.2y^3, \\
    g(t, u, x, y) &= \left(0.28 + 0.04\sin(2\pi t) + 0.04\frac{u}{1+u}\right)x^2 + 0.1y^2,
\end{align*}
and the initial history is $\xi(s) = 1 + 0.1\sin s$ for $s \in [-1/2, 0]$.
\end{example}

For this problem, the polynomial growth degree is $\ell = 4$, and the H\"older regularity parameters in time are $\beta_1 = 1$ and $\beta_2 = 1$.
According to \Cref{thm:fully-discrete}, the theoretical convergence rate is dominated by $\alpha/2$.
We test the stability indices $\alpha = 0.6$ and $\alpha = 0.8$, for which the theoretical reference slopes are $0.3$ and $0.4$, respectively.
The empirical least-squares convergence rates across all tested step sizes are $0.312$ for $\alpha = 0.6$ and $0.388$ for $\alpha = 0.8$.
\Cref{fig:ex3} compares the numerical errors with the theoretical reference slopes.

\begin{figure}[H]
\centering
\includegraphics[width=.48\linewidth]{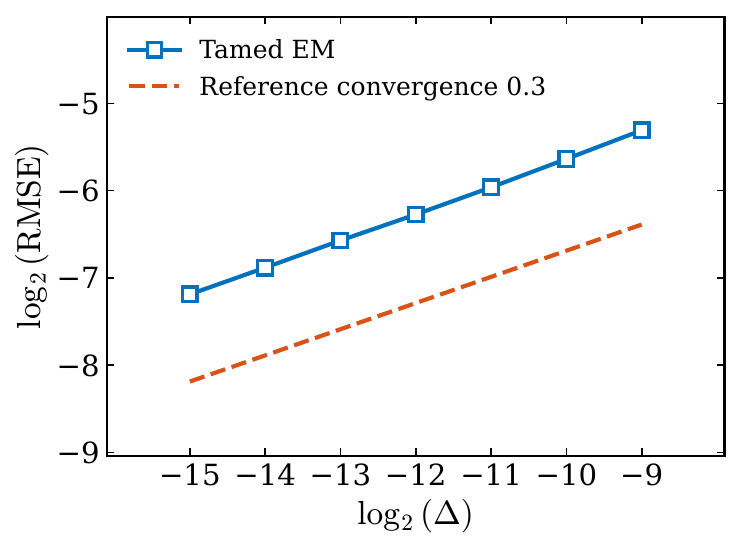}\hfill
\includegraphics[width=.48\linewidth]{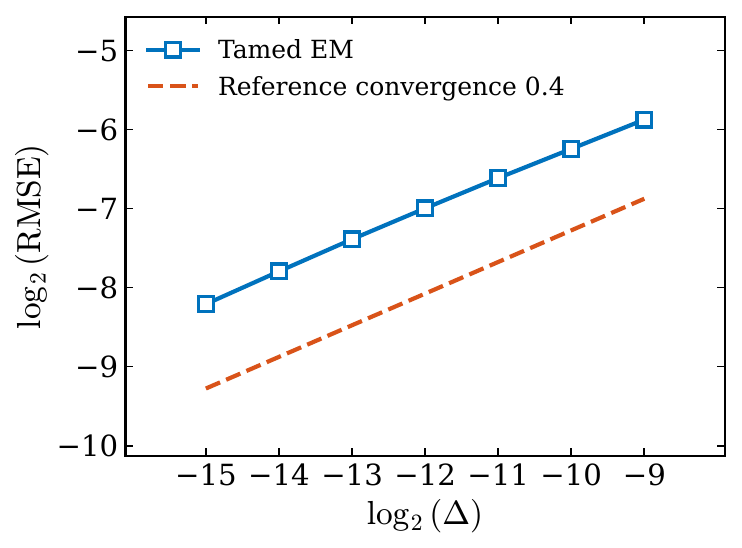}
\caption{Strong convergence rates for \Cref{ex:polynomial}: $\alpha=0.6$ (left) and $\alpha=0.8$ (right).}\label{fig:ex3}
\end{figure}

We next examine the convergence with respect to the subordinator sampling step $\delta$ predicted by \Cref{thm:fully-discrete} for \Cref{ex:polynomial}.
To isolate the sampling error, we fix $\Delta=2^{-22}$ and vary $\delta\in\{2^{-9},\dots,2^{-15}\}$, computing the grid-maximum RMSE against an ETEM reference solution with the same $\Delta$ and $\delta_{\rm ref}=2^{-23}$.
As shown in \Cref{fig:ex3-clock}, the empirical least-squares convergence rates are $0.481$ for $\alpha=0.6$ and $0.443$ for $\alpha=0.8$.

\begin{figure}[H]
\centering
\includegraphics[width=.48\linewidth]{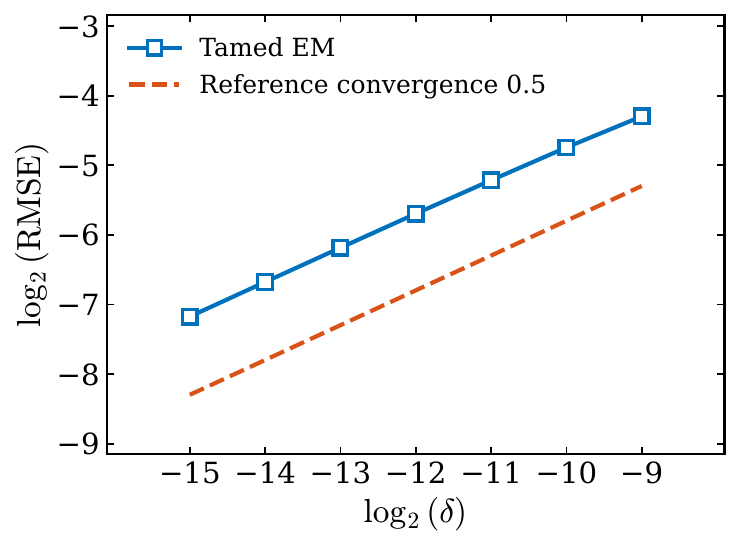}\hfill
\includegraphics[width=.48\linewidth]{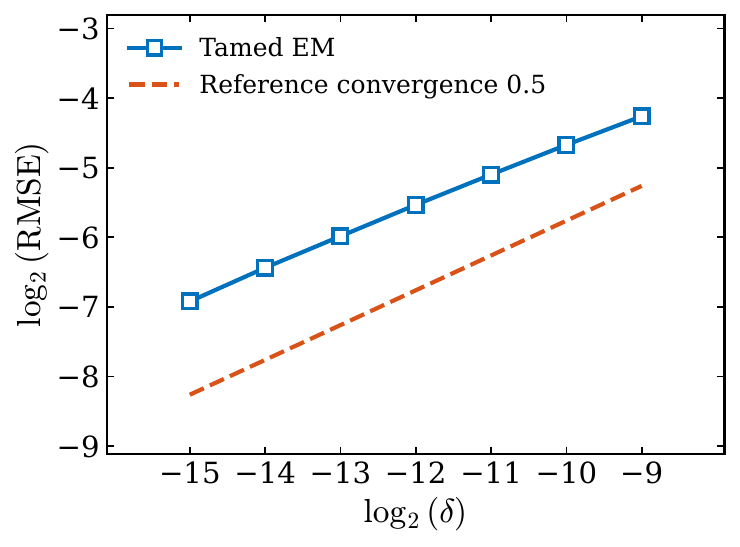}
\caption{Strong convergence rates with respect to $\delta$ for \Cref{ex:polynomial}: $\alpha=0.6$ (left) and $\alpha=0.8$ (right).}\label{fig:ex3-clock}
\end{figure}

\begin{example}\label{ex:2d-delay}
We next consider a two-dimensional TCSDDE with coupled superlinear drifts and cross delays.
Let $X(t) = (X_1(t), X_2(t))^\top$, driven by two independent Brownian motions $B_1, B_2$ sharing the common inverse subordinator $E(t)$.
For $i = 1, 2$ and $j = 3 - i$, the system reads
\begin{align*}
    \mathrm{d}X_i(t) &= \left(-X_i(t) - X_i^3(t) + 0.2X_j(t-\tau)\right)\mathrm{d}t + \left(-2X_i(t) - X_i^5(t) + 0.2X_j(t-\tau)\right)\mathrm{d}E(t) \\
    &\quad + \left(0.2X_j^2(t) + 0.1X_j(t-\tau)\right)\mathrm{d}B_i(E(t)),
\end{align*}
with constant initial history $X(s) = (1, 0.8)^\top$ for $s \in [-1/2, 0]$.
\end{example}

Since the system is autonomous, we may take $\beta_1=\beta_2=1$.
The polynomial exponent in the strong error bound is therefore $\alpha/2$.
For $\alpha = 0.6$ and $\alpha = 0.8$, the theoretical reference slopes are $0.3$ and $0.4$, while the empirical least-squares rates are $0.264$ and $0.377$, respectively.
\Cref{fig:ex4} displays the numerical errors alongside the reference slopes.

\begin{figure}[H]
\centering
\includegraphics[width=.48\linewidth]{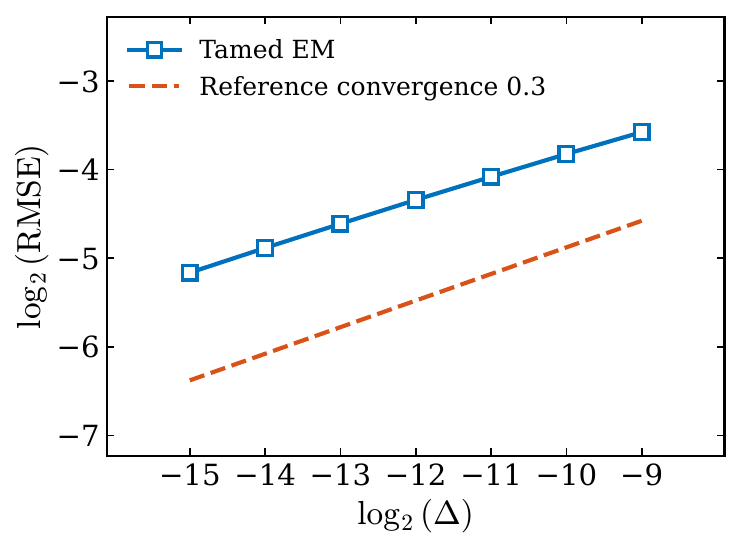}\hfill
\includegraphics[width=.48\linewidth]{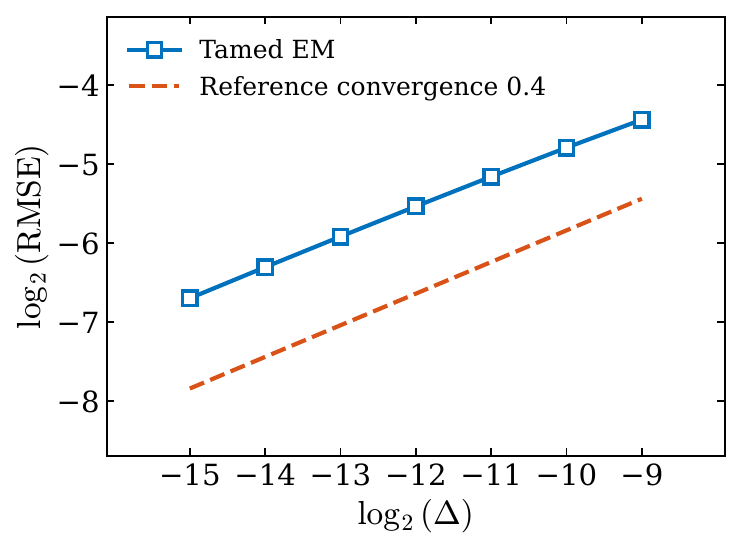}
\caption{Strong convergence rates for \Cref{ex:2d-delay}: $\alpha=0.6$ (left) and $\alpha=0.8$ (right).}\label{fig:ex4}
\end{figure}

\begin{example}\label{ex:rough-coefficients}
To examine the role of the time regularity exponents $\beta_1$ and $\beta_2$ in \Cref{thm:fully-discrete}, we consider a TCSDDE with $\beta$-H\"older continuous coefficients in either $t$ or $E(t)$.
Let $\psi(s) = 1 - 4|\{s + \frac{1}{4}\} - \frac{1}{2}|$ be the periodic triangular wave, where $\{s\}$ denotes the fractional part of $s$.
For $\beta\in(0,1)$, we consider the following Takagi--Landsberg-type function
\begin{equation*}
    W_\beta(s) = (1 - 2^{-\beta}) \sum_{k=0}^\infty 2^{-k\beta} \psi(2^k s),
\end{equation*}
which is bounded and $\beta$-H\"older continuous (see \cite{MishuraSchied2019, MakoginMishura2020} for details).
The equation is given by
\begin{align*}
    \mathrm{d}X(t) &= (-X(t) + 0.2X(t-\tau))\,\mathrm{d}t + (-0.5X(t) + 0.1X(t-\tau))\,\mathrm{d}E(t) \\
    &\quad + g_\bullet(t, E(t), X(t), X(t-\tau))\,\mathrm{d}B(E(t)),
\end{align*}
with initial history $\xi(s) = 1 + 0.1\sin s$ for $s \in [-1/2, 0]$ and fixed $\alpha = 0.8$.
We examine two configurations for the diffusion coefficient
\begin{itemize}
    \item Case A (roughness in $t$):
    \begin{equation*}
        g_A(t, u, x, y) = 0.1x + 0.05y + 0.5 + W_\beta(t) + 0.05\frac{u}{1+u},
    \end{equation*}
    corresponding to $\beta_1 = \beta$ and $\beta_2 = 1$.
    \item Case B (roughness in $E(t)$):
    \begin{equation*}
        g_B(t, u, x, y) = 0.1x + 0.05y + 0.5 + W_\beta(u) + 0.05\sin t,
    \end{equation*}
    corresponding to $\beta_1 = 1$ and $\beta_2 = \beta$.
\end{itemize}
\end{example}

In Case A, the theoretical rate is $\min\{\beta_1, \alpha/2\} = \min\{\beta, 0.4\}$.
For $\beta = 0.25$ and $\beta = 0.35$, the theoretical slopes are $0.25$ and $0.35$, with empirical rates $0.279$ and $0.358$, respectively.
In Case B, since $\beta < 1/2$, the theoretical rate is governed by $\alpha\beta = 0.8\beta$.
For $\beta = 0.25$ and $\beta = 0.35$, the theoretical slopes are $0.20$ and $0.28$, with empirical rates $0.210$ and $0.290$, respectively.
\Cref{fig:ex5} compares the numerical errors with the theoretical reference slopes across all four cases.

\begin{figure}[H]
\centering
\includegraphics[width=.48\linewidth]{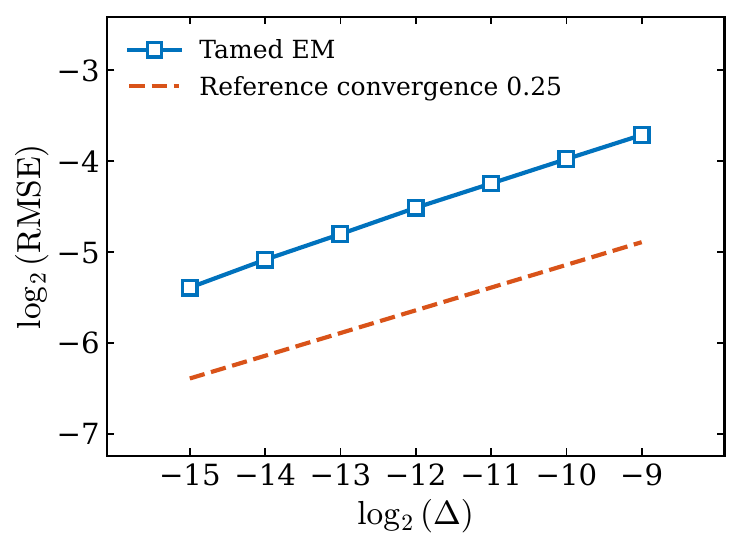}\hfill
\includegraphics[width=.48\linewidth]{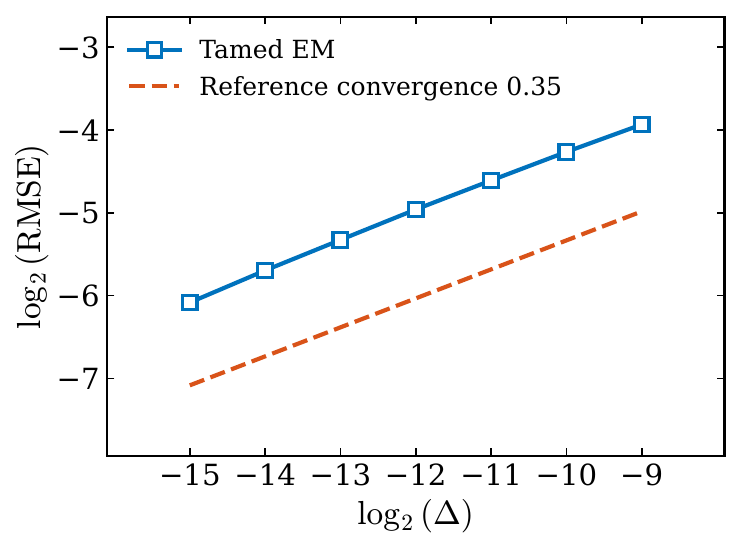}\\[1ex]
\includegraphics[width=.48\linewidth]{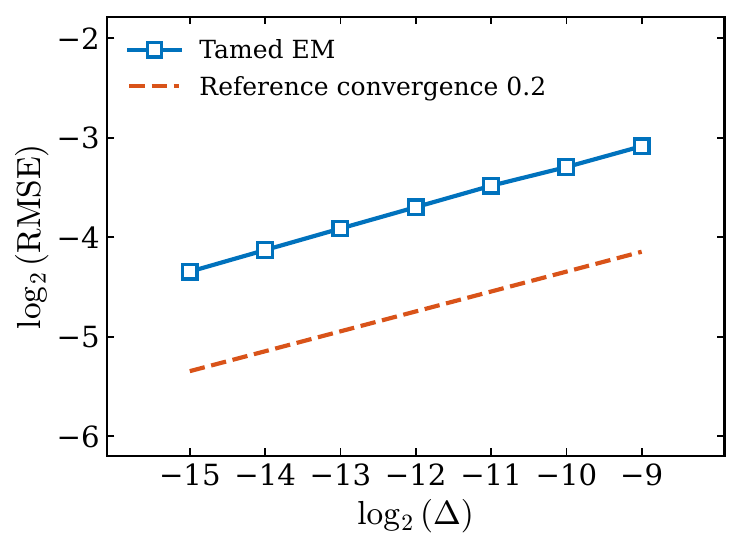}\hfill
\includegraphics[width=.48\linewidth]{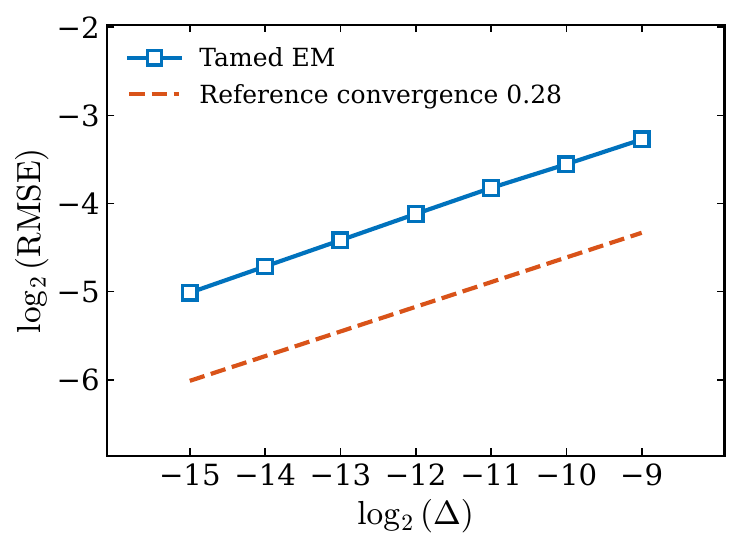}
\caption{Strong convergence rates for \Cref{ex:rough-coefficients}: Case A with $\beta=0.25$ (top-left) and $\beta=0.35$ (top-right); Case B with $\beta=0.25$ (bottom-left) and $\beta=0.35$ (bottom-right).}\label{fig:ex5}
\end{figure}

\section{Conclusion}\label{sec:conclusion}

In this paper, we developed a fully discrete ETEM method for superlinear TCSDDEs with state-dependent $\mathrm{d}t$ drift.
The method combines an equidistant time grid with a sampled inverse subordinator and uses separate taming for the $\mathrm{d}t$ drift and coupled taming for the subordinator-driven drift and diffusion.
Under a coupled one-sided condition, we established strong error bounds in the grid-maximum norm for independently chosen time and subordinator sampling steps, together with a lower bound showing the sharpness of both polynomial exponents.
Numerical experiments illustrated the convergence behavior for nonautonomous, multidimensional, and low-regularity delay models.
Future work includes higher-order strong approximation methods, long-time stability, and the approximation of limiting distributions.

\appendix
\makeatletter
\renewcommand{\@seccntformat}[1]{\appendixname~\csname the#1\endcsname.\quad}
\makeatother

\section{Proof of Lemma~\ref{lem:martingale-modulus}}\label[appendix]{app:martingale-modulus}

\begin{proof}
If $\EB[S^{2p}]=0$, the claim is immediate.
Otherwise, put $M_2=\lceil L/\delta\rceil$.
For $R>0$, let $M_R(u)=\int_0^uG(v)\mathbf1_{\{|G(v)|\leq R\}}\dd B(v)$.
On the event $\{S\leq R\}=\{\sup_{0\leq v\leq L}|G(v)|\leq R\}$, the indicator equals $1$ identically on $[0,L]$, so $M=M_R$ almost surely.
Since every interval $[u,v]\subseteq[0,L]$ with $v-u\leq\delta$ lies in some $[k\delta,((k+2)\delta)\wedge L]$,
\begin{equation}\label{eq:mr-window-bound}
    \sup_{\substack{0\leq u<v\leq L\\v-u\leq\delta}}|M_R(v)-M_R(u)|\leq2\max_{0\leq k<M_2}\sup_{k\delta\leq w\leq((k+2)\delta)\wedge L}|M_R(w)-M_R(k\delta)|.
\end{equation}
For $x>0$, define the target event
\begin{equation*}
    A=\biggl\{\sup_{\substack{0\leq u<v\leq L\\v-u\leq\delta}}|M(v)-M(u)|>x\biggr\}.
\end{equation*}
Partitioning the sample space into $\{S>R\}$ and $\{S\leq R\}$, and noting that $M=M_R$ on $\{S\leq R\}$, we obtain
\begin{align*}
    A&=(A\cap\{S>R\})\cup(A\cap\{S\leq R\})\\
    &\subseteq\{S>R\}\cup\biggl\{\sup_{\substack{0\leq u<v\leq L\\v-u\leq\delta}}|M_R(v)-M_R(u)|>x\biggr\}\\
    &\subseteq\{S>R\}\cup\bigcup_{k=0}^{M_2-1}\biggl\{\sup_{k\delta\leq w\leq((k+2)\delta)\wedge L}|M_R(w)-M_R(k\delta)|>\frac x2\biggr\},
\end{align*}
where the last inclusion follows from \eqref{eq:mr-window-bound}.
Taking the probability $\PP_B$ on both sides and applying the union bound yields
\begin{align}\label{eq:martingale-tail-bound}
    \PP_B(A)
    &\leq\PP_B(S>R)+\sum_{k=0}^{M_2-1}\PP_B\left(\sup_{k\delta\leq w\leq((k+2)\delta)\wedge L}|M_R(w)-M_R(k\delta)|>\frac x2\right)\notag\\
    &\leq\PP_B(S>R)+CM_2\exp\left\{-\frac{x^2}{C\delta R^2}\right\},
\end{align}
where the second inequality applies Bernstein's continuous martingale maximal inequality to each coordinate of $M_R$, whose quadratic variation on each window of length at most $2\delta$ is bounded by $2\delta R^2$.

For $y\geq1$, we choose the scale-dependent parameters
\begin{equation*}
    R=\left(\EB[S^{2p}]\right)^{1/(2p)}\sqrt y,\qquad x=C_1\left(\EB[S^{2p}]\right)^{1/(2p)}\sqrt{\delta\log(e+L/\delta)}\,y.
\end{equation*}
Since $M_2\leq1+L/\delta\leq e+L/\delta$, choosing $C_1$ sufficiently large guarantees that
\begin{equation}\label{eq:exp-tail-bound}
    CM_2\exp\left\{-\frac{x^2}{C\delta R^2}\right\}\leq CM_2\exp\left\{-\frac{C_1^2}{C}y\log(e+L/\delta)\right\}\leq Ce^{-y}.
\end{equation}
Define the normalized random variable
\begin{equation*}
	U=\frac{\sup_{0\leq u<v\leq L,\,v-u\leq\delta}|M(v)-M(u)|}{C_1\left(\EB[S^{2p}]\right)^{1/(2p)}\sqrt{\delta\log(e+L/\delta)}}.
\end{equation*}
By definition of $A$, the event $\{U>y\}$ coincides with $A$ for the threshold $x$ chosen above.
Substituting this $x$, $R=\left(\EB[S^{2p}]\right)^{1/(2p)}\sqrt y$, and \eqref{eq:exp-tail-bound} into the tail bound \eqref{eq:martingale-tail-bound}, we obtain
\begin{equation}\label{eq:u-tail-bound}
    \PP_B(U>y)\leq\PP_B\left(\frac{S}{\left(\EB[S^{2p}]\right)^{1/(2p)}}>\sqrt y\right)+Ce^{-y},\qquad y\geq1.
\end{equation}
To estimate the $p$th moment, we express $\EB[U^p]$ via tail integration and split the integral at $y=1$,
\begin{equation*}
    \EB[U^p]=p\int_0^1 y^{p-1}\PP_B(U>y)\dd y+p\int_1^\infty y^{p-1}\PP_B(U>y)\dd y.
\end{equation*}
Since $\PP_B(U>y)\leq1$ on $[0,1]$, the first integral satisfies $\int_0^1 py^{p-1}\PP_B(U>y)\dd y\leq\int_0^1 py^{p-1}\dd y=1$.
Substituting the tail bound \eqref{eq:u-tail-bound} into the second integral yields
\begin{align*}
    \EB[U^p]
    &\leq1+p\int_1^\infty y^{p-1}\PP_B\left(\frac{S}{\left(\EB[S^{2p}]\right)^{1/(2p)}}>\sqrt y\right)\dd y+Cp\int_1^\infty y^{p-1}e^{-y}\dd y\\
    &\leq1+\int_0^\infty 2pz^{2p-1}\PP_B\left(\frac{S}{\left(\EB[S^{2p}]\right)^{1/(2p)}}>z\right)\dd z+C\\
    &=1+\EB\left[\left(\frac{S}{\left(\EB[S^{2p}]\right)^{1/(2p)}}\right)^{2p}\right]+C\leq C,
\end{align*}
where the second inequality uses the substitution $z=\sqrt y$ ($py^{p-1}\dd y=2pz^{2p-1}\dd z$) and $\int_1^\infty y^{p-1}e^{-y}\dd y\leq\Gamma(p)$.
Rescaling $U$ gives the asserted moment bound.

\end{proof}

\section{Proof of Proposition~\ref{prop:exact}}\label[appendix]{app:exact-moments}

\begin{proof}
Fix a subordinator path from the full-measure set specified above and work under $\PP_B$.
Under $\PP_B$, $E$ is a continuous nondecreasing path and $B(E(t))$ is a continuous local martingale.
By \Cref{ass:polynomial}, the coefficients $h, f, g$ are locally Lipschitz continuous in $(x,y)$.
The classical localization theorem for stochastic delay differential equations (see, e.g., \cite{Mao1991, Protter2004, Zuo2026}) guarantees that \eqref{eq:model} admits a pathwise unique maximal local strong solution $X(t)$ on $[0, \tau_\infty)$, where $\tau_\infty$ denotes the explosion time.

To show that the solution is nonexplosive on $[0,T]$ (that is, $\tau_\infty>T$ almost surely) and to establish the uniform moment bounds, we decompose each coefficient at the current state $x'=0$.
The coupled condition in \Cref{ass:current} controls the drift and diffusion terms together.
For every fixed $p\geq2$, using the coupled condition in \Cref{ass:current} with weight $2p$ and Young's inequality gives
\begin{align}
    \ip{x}{h(t,u,x,y)}&\leq C(1+|x|^2+|y|^{2(\ell+1)}),\notag\\
    \ip{x}{f(t,u,x,y)}+p|g(t,u,x,y)|^2&\leq C(1+|x|^2+|y|^{2(\ell+1)}).\label{eq:energy}
\end{align}
Indeed, writing $f_0=f(t,u,0,y)$ and $g_0=g(t,u,0,y)$, Young's inequality $\ip{x}{f_0}\leq\frac12|x|^2+\frac12|f_0|^2$ and \Cref{ass:current} give
\begin{align*}
    &\quad\ip{x}{f(t,u,x,y)}+p|g(t,u,x,y)|^2\\
    &\leq\ip{x}{f(t,u,x,y)-f_0}+2p|g(t,u,x,y)-g_0|^2+\ip{x}{f_0}+2p|g_0|^2\\
    &\leq C|x|^2+\tfrac12|x|^2+\tfrac12|f_0|^2+2p|g_0|^2\\
    &\leq C(1+|x|^2+|y|^{2(\ell+1)}).
\end{align*}
The estimate for $\ip{x}{h}$ in \eqref{eq:energy} follows similarly.

Set $\zeta_R=\inf\{t\in[0,T]\colon|X(t)|\geq R\}$, with $\inf\varnothing=\infty$.
The stopped process and its delayed states before $\zeta_R$ are bounded, so all stopped stochastic integrals below are true martingales.
Applying It\^o's formula, \eqref{eq:energy}, and Young's inequality, we obtain
\begin{equation}\label{eq:exact-ito-decomp}
    |X(t\wedge\zeta_R)|^p+\frac{p(p+1)}2J_R(t)\leq F_R(t)+\mathcal{M}_R(t),
\end{equation}
where
\begin{align*}
    J_R(t)&=\int_0^{t\wedge\zeta_R}|X(s)|^{p-2}|g(s,E(s),X(s),X(s-\tau))|^2\dd E(s),\\
    F_R(t)&=|\xi(0)|^p+C\int_0^{t\wedge\zeta_R}(1+|X(s)|^p+|X(s-\tau)|^{(\ell+1)p})\dd A(s),\\
    \mathcal{M}_R(t)&=p\int_0^{t\wedge\zeta_R}|X(s)|^{p-2}X(s)^\top g(s,E(s),X(s),X(s-\tau))\dd B(E(s)).
\end{align*}
For $p=2$, the factor $|X(s)|^{p-2}$ is one.
Since $\mathcal{M}_R$ is a continuous martingale with $\EB[\mathcal{M}_R(t)]=0$, dropping the non-negative term $|X(t\wedge\zeta_R)|^p$ and taking $\EB$ in \eqref{eq:exact-ito-decomp} yields
\begin{equation}\label{eq:exact-jr-bound}
    \EB[J_R(t)]\leq C\EB[F_R(t)].
\end{equation}
By the Burkholder--Davis--Gundy (BDG) and Young inequalities,
\begin{align}\label{eq:exact-mr-bound}
    \EB\left[\sup_{0\leq u\leq t}|\mathcal{M}_R(u)|\right]
    &\leq C\EB\left[\left\{\sup_{0\leq u\leq t}|X(u\wedge\zeta_R)|^pJ_R(t)\right\}^{1/2}\right]\notag\\
    &\leq\tfrac12\EB\left[\sup_{-\tau\leq u\leq t}|X(u\wedge\zeta_R)|^p\right]+C\EB[J_R(t)]\notag\\
    &\leq\tfrac12\EB\left[\sup_{-\tau\leq u\leq t}|X(u\wedge\zeta_R)|^p\right]+C\EB[F_R(t)],
\end{align}
where the third inequality uses \eqref{eq:exact-jr-bound}.
Dropping the non-negative energy term $\frac{p(p+1)}{2}J_R$ in \eqref{eq:exact-ito-decomp}, taking the supremum over $u\in[0,t]$, and combining with the initial history $\sup_{-\tau\leq u\leq0}|X(u\wedge\zeta_R)|^p\leq C$, we obtain
\begin{equation*}
    \sup_{-\tau\leq u\leq t}|X(u\wedge\zeta_R)|^p\leq C+F_R(t)+\sup_{0\leq u\leq t}|\mathcal{M}_R(u)|.
\end{equation*}
Taking $\EB$ on both sides and substituting \eqref{eq:exact-mr-bound}, we absorb the resulting $\frac12$-weighted supremum moment into the left-hand side.
For every $0\leq t\leq k\tau$ with $1\leq k\leq K$, this yields
\begin{align}\label{eq:exact-integral-ineq}
    &\quad\EB\left[\sup_{-\tau\leq u\leq t}|X(u\wedge\zeta_R)|^p\right]\notag\\
    &\leq C+C\int_0^t\EB\left[\sup_{-\tau\leq u\leq s}|X(u\wedge\zeta_R)|^p\right]\dd A(s)+C\int_0^t\EB\left[\mathbf 1_{\{s<\zeta_R\}}|X(s-\tau)|^{(\ell+1)p}\right]\dd A(s)\notag\\
    &\leq C+C\int_0^t\EB\left[\sup_{-\tau\leq u\leq s}|X(u\wedge\zeta_R)|^p\right]\dd A(s)+CA(t)\EB\left[\sup_{-\tau\leq u\leq(k-1)\tau}|X(u\wedge\zeta_R)|^{(\ell+1)p}\right],
\end{align}
where the second inequality uses $s-\tau<\zeta_R$ on $\{s<\zeta_R\}$ to write $X(s-\tau)=X((s-\tau)\wedge\zeta_R)$, applies \Cref{lem:delay} to bound the delayed state by the history supremum on $[-\tau,(k-1)\tau]$, and retains $\int_0^t\mathbf 1_{\{s<\zeta_R\}}\dd A(s)\leq A(t)$.

Applying \Cref{lem:stieltjes-gronwall} to \eqref{eq:exact-integral-ineq} on $[0,k\tau]$ yields
\begin{equation}\label{eq:exact-gronwall-step}
    \EB\left[\sup_{-\tau\leq u\leq k\tau}|X(u\wedge\zeta_R)|^p\right]
    \leq Ce^{CA(T)}\left[1+A(T)\EB\left[\sup_{-\tau\leq u\leq(k-1)\tau}|X(u\wedge\zeta_R)|^{(\ell+1)p}\right]\right].
\end{equation}
For $k=0$, the bounded initial history gives $\EB[\sup_{-\tau\leq u\leq0}|X(u\wedge\zeta_R)|^p]=\sup_{-\tau\leq u\leq0}|\xi(u)|^p\leq C$.
Induction on $k=1,\ldots,K$ via \eqref{eq:exact-gronwall-step}, simultaneously for every finite $p\geq2$, therefore yields
\begin{equation}\label{eq:exact-stopped-uniform}
    \EB\left[\sup_{-\tau\leq u\leq T}|X(u\wedge\zeta_R)|^p\right]\leq Ce^{CA(T)},
\end{equation}
where for a fixed $p$, only the finitely many orders $p,(\ell+1)p,\ldots,(\ell+1)^Kp$ are needed, and polynomial factors in $A(T)$ are absorbed into the exponential.

By Chebyshev's inequality and \eqref{eq:exact-stopped-uniform} at $p=2$,
\begin{equation*}
    \PP_B(\zeta_R\leq T)\leq R^{-2}\EB\left[\sup_{-\tau\leq u\leq T}|X(u\wedge\zeta_R)|^2\right]\leq CR^{-2}e^{CA(T)}\longrightarrow0\qquad\text{as } R\to\infty.
\end{equation*}
Since $\zeta_R \uparrow \tau_\infty$ almost surely as $R\to\infty$, this proves that $\tau_\infty > T$ almost surely under $\PP_B$, and hence the maximal local solution does not explode on $[0,T]$.
Fatou's lemma applied to \eqref{eq:exact-stopped-uniform} as $R\to\infty$ then yields the asserted conditional moment bound
\begin{equation}\label{eq:exact-cond-moment}
    \EB\left[\sup_{-\tau\leq t\leq T}|X(t)|^p\right]\leq Ce^{CA(T)}.
\end{equation}
Finally, taking total expectation with respect to $D$ on both sides of \eqref{eq:exact-cond-moment}, recalling $A(T)=T+E(T)$, and applying the exponential integrability from \Cref{lem:inverse}, we have
\begin{equation*}
    \EE\left[\sup_{-\tau\leq t\leq T}|X(t)|^p\right]\leq C\ED[e^{C(T+E(T))}]\leq C,
\end{equation*}
which completes the proof.
\end{proof}

\section*{Statements and Declarations}
\noindent\textbf{Funding.} Qian Guo was supported by the National Natural Science Foundation of China under Grant No.~12271368.
\par\noindent\textbf{Competing interests.} The authors have no relevant financial or non-financial interests to disclose.
\par\noindent\textbf{Data availability.} The data generated during the current study are available from the first author upon reasonable request.
\par\noindent\textbf{Declaration of AI use.} The authors used ChatGPT to assist with English language editing and improving readability.
All suggested changes were reviewed and verified by the authors.
The authors take full responsibility for the content of this manuscript.

\begingroup
\small
\bibliographystyle{plain}
\bibliography{refs}
\endgroup

\end{document}